\documentclass[a4paper,10pt, oneside]{article}
\usepackage{amsmath,amssymb,amsthm,bm,bbm,mathrsfs,mathtools,graphicx}
\usepackage{algorithm,algorithmic}
\usepackage[a4paper,paperheight=297mm,paperwidth=210mm,margin=1.25in]{geometry} 
\usepackage{authblk}
\usepackage[font=small,labelfont=md,textfont=it]{caption}
\usepackage[titletoc, title]{appendix}
\usepackage{enumitem}
\usepackage[colorlinks,linkcolor=black,citecolor=black]{hyperref}
\usepackage{etoolbox}
\usepackage{longtable}
\usepackage{diagbox}
\usepackage{booktabs,makecell,multirow}
\usepackage{cases,color}
\usepackage[most]{tcolorbox}

\usepackage[capitalise]{cleveref}

\crefname{assumption}{Assumption}{Assumptions}
\crefname{hypothese}{Hypothese}{Hypotheses}
\crefname{lemma}{Lemma}{Lemmas}
\crefname{theorem}{Theorem}{Theorems}
\crefname{discr}{Discretization}{Discretizations}

\crefformat{equation}{\textup{#2(#1)#3}}
\crefrangeformat{equation}{\textup{#3(#1)#4--#5(#2)#6}}
\crefmultiformat{equation}{\textup{#2(#1)#3}}{ and \textup{#2(#1)#3}}
{, \textup{#2(#1)#3}}{, and \textup{#2(#1)#3}}
\crefrangemultiformat{equation}{\textup{#3(#1)#4--#5(#2)#6}}%
{ and \textup{#3(#1)#4--#5(#2)#6}}{, \textup{#3(#1)#4--#5(#2)#6}}{, and \textup{#3(#1)#4--#5(#2)#6}}

\Crefformat{equation}{#2Equation~\textup{(#1)}#3}
\Crefrangeformat{equation}{Equations~\textup{#3(#1)#4--#5(#2)#6}}
\Crefmultiformat{equation}{Equations~\textup{#2(#1)#3}}{ and \textup{#2(#1)#3}}
{, \textup{#2(#1)#3}}{, and \textup{#2(#1)#3}}
\Crefrangemultiformat{equation}{Equations~\textup{#3(#1)#4--#5(#2)#6}}%
{ and \textup{#3(#1)#4--#5(#2)#6}}{, \textup{#3(#1)#4--#5(#2)#6}}{, and \textup{#3(#1)#4--#5(#2)#6}}

\crefdefaultlabelformat{#2\textup{#1}#3}

\apptocmd{\sloppy}{\hbadness 10000\relax}{}{}
\newcommand{\dual}[1]{\big\langle {#1} \big\rangle}

\newcommand{\norm}[1]{\lVert {#1} \rVert}
\newcommand{\normb}[1]{\big\lVert {#1} \big\rVert}

\newcommand{\ssnm}[1]
{
	\left\vert\kern-0.25ex
	\left\vert\kern-0.25ex
	\left\vert
	{#1}
	\right\vert\kern-0.25ex
	\right\vert\kern-0.25ex
	\right\vert
}

\makeatletter
\def\spher@harm#1{%
	\vbox{\hbox{%
			\offinterlineskip
			\valign{&\hb@xt@2\p@{\hss$##$\hss}\vskip.2ex\cr#1\crcr}%
		}\vskip-.36ex}%
}
\def\gshone{\spher@harm{.}}
\def\gshtwo{\spher@harm{.&.}}
\def\gshthree{\spher@harm{.&.&.}}
\let\gsh\spher@harm
\makeatother

\newtheorem{proposition}{Proposition}[section]

\newtheorem{hypothesis}{Hypothesis}[section]
\newtheorem{lemma}{Lemma}[section]
\newtheorem{remark}{Remark}[section]
\newtheorem{theorem}{Theorem}[section]

\numberwithin{equation}{section}

\makeatletter\def\@captype{table}\makeatother

\begin{document}


\title{\Large\bf Pathwise uniform convergence of numerical approximations for a two-dimensional stochastic Navier-Stokes equation with no-slip boundary conditions\thanks{This work was partially supported by the National Natural Science Foundation of China under grants 12301525 and 12171340, and by the Natural Science Foundation of Sichuan Province under grant 2023NSFSC1324.}}
\author[1]{Binjie Li\thanks{libinjie@scu.edu.cn}}
\author[1]{Xiaoping Xie\thanks{xpxie@scu.edu.cn}}
\author[2]{Qin Zhou\thanks{zqmath@cwnu.edu.cn}}
\affil[1]{School of Mathematics, Sichuan University, Chengdu 610064, China}
\affil[2]{School of Mathematics, China West Normal University, Nanchong 637002, China}
\date{}

\maketitle

\begin{abstract}
This paper investigates the pathwise uniform convergence in probability of fully discrete finite-element
approximations for the two-dimensional stochastic Navier-Stokes equations with multiplicative noise,
subject to no-slip boundary conditions.
We demonstrate that the full discretization
achieves nearly $ 3/2$-order convergence in space and nearly half-order convergence in time.
\end{abstract}




\section{Introduction}

Let \(0 < T < \infty\), and consider \(\mathcal{O} \subset \mathbb{R}^2\) as a bounded, connected domain with a sufficiently smooth boundary \(\partial\mathcal{O}\). We study the following incompressible stochastic Navier-Stokes equation (SNSE):
\begin{small}
\begin{equation}
  \label{eq:model}
  \begin{cases}
    \begin{aligned}
      \mathrm{d}y(t,x) &= \Bigl[ \Delta y(t,x) - (y(t,x) \cdot \nabla)y(t,x) - \nabla \varphi(t,x) \Bigr]\,\mathrm{d}t \\
                       &\quad + \sum_{n\in\mathbb{N}} \bigl[ f_n(x,y(t,x)) - \nabla\widetilde\varphi_n(t,x) \bigr] \, \mathrm{d}\beta_n(t), && t\in[0,T], \, x \in \mathcal{O}, \\
      \nabla \cdot y(t,x) &= 0, && t\in[0,T], \, x \in \mathcal{O}, \\
      y(t,x) &= 0, && t \in [0,T], \, x \in \partial\mathcal{O}, \\
      y(0,x) &= y_0(x), && x \in \mathcal{O}.
    \end{aligned}
  \end{cases}
\end{equation}
\end{small}
\hskip -.3em In this system, \(y\) represents the velocity field of interest, \(\varphi\) denotes the pressure field, and \((\widetilde\varphi_n)_{n\in\mathbb N}\) is a sequence of unknown auxiliary scalar processes that ensures the divergence-free property of the diffusion term. The initial velocity field is given by \(y_0\), \((\beta_n)_{n \in \mathbb{N}}\) is a sequence of independent Brownian motions, and the diffusion coefficients \((f_n)_{n \in \mathbb{N}}\) will be detailed later.

Due to their important applications to the theory of turbulent flows and its significance from
a pure mathematics perspective, the SNSEs have garnered significant interest.
These equations have been extensively studied in the literature;
see \cite{Bensoussan1973,Brzezniak1991,Capinski1991,Flandoli2023,Menaldi2002} and the reference therein.
Consequently, over the past decades, the numerical analysis of SNSEs,
as well as their simplified variants, the stochastic Stokes equations,
has attracted considerable interest from the numerical community.

For the cases of space-periodic boundary conditions, significant advancements have been observed in recent years.
We briefly summarize some recent contributions as follows.
D\"orsek~\cite{Dorsek2012} derived rates of convergence for full discretizations that
use spectral Galerkin method in space and semigroup splitting or cubature methods in time,
for a two-dimensional SNSE with additive noise.
In the context of finite-element-based full discretizations,
Carelli and Prohl~\cite{Prohl2012} made notable progress by establishing linear convergence in
space and almost $\frac{1}{4}$-order convergence in time with respect to convergence in probability.
Their findings were later improved upon by Breit and Dodgson~\cite{Breit2021},
who demonstrated enhanced theoretical results with linear convergence in space and
nearly $\frac{1}{2}$-order convergence in time.
Bessaih et al.~\cite{Bessaih2014} derived almost $\frac{1}{2}$-order temporal convergence in probability for
a splitting up method.
Bessaih and Millet \cite{Bessaih2019} established strong $ L^2 $ convergence for both the Euler scheme and
a splitting scheme, and they subsequently extended their analysis to a full discretization that
 combines the Euler scheme for temporal approximation with the finite element method for spatial approximation,
proving strong $L^2$ convergence once more \cite{Bessaih2022}.
Hausenblas and Randrianasolo \cite{Hausenblas2019} studied a time-discretization
with penalty-projection method, and derived a convergence rate of $\frac14$ in probability.
A recent contribution by Breit et al.~\cite{Breit2024} derived a convergence rate of $\frac{1}{2}$
in the mean square error for a time-discretization of a two-dimensional SNSE with transport noise.
We also direct the reader to \cite{Feng2021,FengQiu2021} and the references therein for the numerical analysis
of stochastic Stokes equations.
Of particular note is the recent work by Li et al.~\cite{LiMaSun2024}, which established a strong convergence rate of
$ O(h^2 + \tau^{1/2}) $ in the $ L^\infty(0,T;L^2(\Omega;L^2)) $ norm for finite-element-based approximations
of two- and three-dimensional stochastic Stokes equations under the stress boundary conditions.

Despite significant progress in this field, the numerical analysis of SNSEs subject to no-slip
boundary conditions remains a relatively underdeveloped area of research. A fundamental difficulty
arises from the fact that the Helmholtz projection \(\mathcal{P}\) fails to preserve the
zero-trace property. This limitation implies that the range of \(\mathcal{P}\) is not
necessarily contained within the domain of \(A_2^{1/4}\) (let alone \(A_2^{1/2}\)),
where \(A_2\) represents the realization of the negative Stokes operator in \(L^2(\mathcal{O}; \mathbb{R}^2)\)
with homogeneous Dirichlet boundary conditions.
Moreover, the important identity $ \int_{\mathcal O} \big((u\cdot\nabla) u\big) \cdot A_2 u \, \mathrm{d}x = 0 $,
which holds on the domain of the Stokes operator under space-periodic boundary conditions,
fails to hold under no-slip boundary conditions.
These mathematical obstacles, combined with the temporal irregularities induced by the
Brownian motions, substantially diminishes the solution's regularity when compared to
scenarios with space-periodic boundary conditions.
Specifically, in the model problem \cref{eq:model},
the trajectories of the process \(y\) generally do not possess
\(L^2(0,T; \dot H^{3/2,2})\)-regularity, where \(\dot H^{3/2,2}\) denotes
the domain of \(A_2^{3/4}\), unless the critical structural assumptions—that the diffusion coefficients are divergence-free and vanish on the boundary \(\partial\mathcal{O}\)—are satisfied.

Here is a brief summary of some recent works on the numerical analysis of SNSEs with no-slip boundary conditions:
\begin{itemize}
  \item Breckner \cite{Breckner2000} demonstrated the convergence in mean square for a spectral-Galerkin spatial semidiscretization.
  \item Brzeźniak and Prohl \cite{Prohl2013} addressed both two- and three-dimensional cases,
showing the convergence of finite element approximations to a unique strong solution
in two dimensions using compactness arguments.
  \item Doghman \cite{Doghman2024} extended the work \cite{Prohl2013} by proving the convergence of finite element
approximations for the two-dimensional SNSEs through the artificial compressibility method.
  \item Ondreját et al.~\cite{Ondrejat2023} achieved convergence for a class of finite element approximations
to a weak martingale solution in three dimensions.
  \item To the best of our knowledge, Breit and Prohl \cite{Breit2023} were the
    first to report strong convergence rates in probability for a fully discrete scheme applied to two-dimensional SNSEs
with multiplicative noise, yielding the following convergence result:
\[
\lim_{\substack{h \to 0 \\ \tau \to 0}}
\mathbb{P}\left[
\frac{
\max_{m} \|\bm{u}(t_{m}) - \bm{u}_{h,m}\|_{L_{x}^{2}}^{2}
+ \sum_{m} \tau \|\nabla(\bm{u}(t_m) - \bm{u}_{h,m})\|_{L_{x}^2}^2
}{h^{\alpha} + \tau^{\beta}} > \xi
\right] = 0,
\]
for any $\alpha \in (1,2)$, $\beta \in (0,1)$, and $\xi > 0$.
\end{itemize}
However, the numerical analysis presented in \cite{Breit2023} is
generally not applicable to the model problem \cref{eq:model},
as it relies on the aforementioned structural assumption that the diffusion coefficients
are divergence-free and vanish on the boundary $ \partial\mathcal O $.

The model problem \cref{eq:model} and its more complex variants, where the diffusion coefficients depend on the gradient
of \( y \), have been extensively studied; see, for instance,
\cite{Brzezniak2020,Mikulevicius2004,Mikulevicius2005,Neerven2012b}.
Despite considerable research efforts, to the best of our knowledge,
rigorous numerical analyses of finite element approximations directly applicable to the model problem \cref{eq:model} are limited to those presented in the aforementioned works \cite{Prohl2013, Doghman2024}.
However, these studies have not provided explicit convergence rates in probability.
This significant gap in the existing literature motivates our current research.

The remainder of this manuscript is organized as follows.
\begin{enumerate}
  \item Section \ref{sec:preliminary} introduces the fundamental notations and presents the concept of global mild solutions.
  \item In Section \ref{sec:regularity}, we conduct a rigorous investigation into the regularity properties of the model problem \cref{eq:model}, which forms the cornerstone of our subsequent numerical analysis. Building upon the theoretical foundations established in \cite{Neerven2008,Neerven2012b}, we prove that the global mild solution of \cref{eq:model} satisfies the probabilistic estimate
    \[
      \mathbb{P}\left(
        \norm{y}_{C([0,T];\dot{H}^{\varrho,2})} \leqslant R
      \right) \geqslant 1 - \frac{c}{\ln(1+R)}, \quad \forall R > 1,
    \]
    where \(1 < \varrho < \frac{3}{2}\), \(\dot{H}^{\varrho,2}\) denotes the domain of \(A_2^{\varrho/2}\), and \(c > 0\) represents a constant independent of \(R\).
    Additionally, the regularity of the global mild solution within the general spatial $ L^q$-norms
    is also considered. The theoretical results are presented in \cref{prop:regu}.
  \item Section \ref{sec:spatial_discretization} focuses on spatial semidiscretization
    using the \(P_3/P_2\) Taylor-Hood finite element method \cite{Stenberg1990}.
    We establish a probabilistic stability estimate and derive both localized and
    global pathwise uniform convergence estimates,
    formalized in \cref{thm:yh-regu,thm:yR-yhR,thm:y-yh-global}.
  \item In Section \ref{sec:full-discretization}, we provide a rigorous analysis of the
    pathwise uniform convergence in probability for a full discretization, which
    uses the aforementioned $P_3/P_2$ Taylor-Hood finite element method and the Euler scheme.
    Our approach follows the methodology developed in \cite{Breit2021,Breit2023,Printems2001},
    utilizing suitably constructed stopping times to partition the probability
    space $\Omega$ into regular and irregular subsets.
    Within the regular subset, the numerical solution demonstrates
    robust convergence as both the spatial mesh size $h$ and the time step $\tau$
    decrease towards zero. Concurrently, the probability measure of the irregular subset
    converges to zero as $h$ and $\tau$ diminish.
  \item Finally, Section \ref{sec:conclusion} summarizes our principal theoretical
    contributions and discusses potential extensions of this work.
\end{enumerate}




\section{Mathematical Setup}
\label{sec:preliminary}
We begin by introducing some notational conventions that will be used throughout this paper.
For any Banach spaces $ E_1 $ and $ E_2 $, the space of all bounded linear operators
from $ E_1 $ to $ E_2 $ is denoted by $ \mathcal L(E_1,E_2) $.
The identity operator is denoted by $ I $.
For any $ \theta \in (0,1) $ and $ p \in (1,\infty) $, we use the notation $ (\cdot,\cdot)_{\theta,p} $
to refer to the space constructed by the real interpolation method (see, e.g., \cite[Chapter~1]{Lunardi2018}).
 Given an interval \( D \subset \mathbb{R} \) and a Banach space \( X \),
 the notation \( C(D; X) \) denotes the space of all continuous functions
 defined on \( D \) with values in \( X \).
Let $ T > 0 $ be a fixed terminal time, and let $\mathcal{O}$ be a bounded domain in $\mathbb{R}^2$ with a smooth boundary
of class $ \mathcal C^{3,1} $ (see \cite[Chapter III, Section~1.1]{Boyer2012}).
For any $ q \in [1,\infty] $, we denote by $ W^{2,q}(\mathcal O) $ and $ W^{2,q}(\mathcal O;\mathbb R^2) $
the standard Sobolev spaces (see \cite[Chapter~5]{Tartar2007}).
For $\alpha \in [0, 2]$ and $q \in (1, \infty)$, define
\[
H^{\alpha,q} := [L^q(\mathcal O), W^{2,q}(\mathcal{O})]_{\alpha/2} \quad\text{and}\quad 
\mathbb{H}^{\alpha,q} := [\mathbb{L}^q(\mathcal O;\mathbb R^2), W^{2,q}(\mathcal{O};\mathbb{R}^2)]_{\alpha/2},
\]
where $[\cdot,\cdot]_{\alpha/2} $ denotes the complex interpolation method (see \cite[Chapter~2]{Lunardi2018}).
Let $\mathcal{H}^{\alpha,q}$ denote the domain of $(-\Delta_q)^{\alpha/2}$,
equippped with the norm
\[
  \norm{u}_{\mathcal H^{\alpha,q}} := \norm{(-\Delta_q)^{\alpha/2}u}_{L^2(\mathcal O;\mathbb R^2)},
  \quad \forall u \in \mathcal H^{\alpha,q},
\]
where $\Delta_q$ is the Dirichlet Laplacian on $L^q(\mathcal O;\mathbb R^2)$.
Its dual space $\mathcal{H}^{-\alpha,q'}$ (with $q' := q/(q-1)$) is equipped
with the duality pairing $\dual{\cdot,\cdot}_{\mathcal{H}^{-\alpha,q'}, \mathcal{H}^{\alpha,q}}$.
The notation $\langle \cdot, \cdot \rangle$ is used for inner products in $L^2(\mathcal{O})$, $L^2(\mathcal{O};\mathbb{R}^2)$, or $L^2(\mathcal{O};\mathbb{R}^{2\times 2})$, depending on the context.
For brevity, we denote $L^q(\mathcal{O})$ by $L^q$. Similarly, $L^q(\mathcal{O};\mathbb{R}^2)$ and $L^q(\mathcal{O};\mathbb{R}^{2\times 2})$ are abbreviated as $\mathbb{L}^q$, as appropriate.

\subsection{Stokes Operator}
For any $q \in (1, \infty)$, we define $\mathbb{L}_{\text{div}}^q$ as the closure in $\mathbb{L}^q$ of the set:
\[
  \left\{ u \in C_{\mathrm{c}}^\infty(\mathcal{O}; \mathbb{R}^2) : \, \nabla \cdot u = 0 \text{ in } \mathcal{O} \right\},
\]
where $C_{\mathrm{c}}^\infty(\mathcal{O}; \mathbb{R}^2)$ denotes the space of vector-valued
functions that are infinitely differentiable with compact support in $\mathcal{O}$.
Let $\mathcal{P}$ denote the Helmholtz projection operator, which projects elements from $\mathbb{L}^q$ onto $\mathbb{L}_{\text{div}}^q$ for $q \in (1, \infty)$. Notably, when $q = 2$, $\mathcal{P}$ acts as the orthogonal projection in $\mathbb{L}^2$ onto $\mathbb{L}_{\text{div}}^2$.
For a detailed discussion of the Helmholtz projection operator, see Chapter III of \cite{Galdi2011}.

For any $q \in (1, \infty)$, the negative Stokes operator $A_q$ is defined as
$ A_q := -\mathcal{P} \Delta_q $, with domain
\[
  \text{D}(A_q) := \mathbb{L}_\text{div}^q \cap \{u \in W^{2,q}(\mathcal{O}; \mathbb{R}^2): \, u = 0 \text{ on } \partial\mathcal{O}\},
\]
where $\Delta_q$ is the realization of the Laplacian in $\mathbb{L}^q$ with homogeneous Dirichlet boundary conditions.
For any $\alpha \in [0, 2]$, we denote by $\dot{H}^{\alpha,q}$ the domain of $A_q^{\alpha/2}$, equipped with the norm:
\[
  \| u \|_{\dot{H}^{\alpha,q}} := \| A_q^{\alpha/2} u \|_{\mathbb{L}^q}, \quad \forall u \in \dot{H}^{\alpha,q}.
\]
By \cite[Theorem~2]{Giga1985}, $\dot{H}^{\alpha,q}$ coincides with the complex interpolation space
$[\mathbb L_{\text{div}}^q, \dot{H}^{2,q}]_{\alpha/2}$, with equivalent norms.
The dual space of $ \dot H^{\alpha,q} $ is denoted by $ \dot H^{-\alpha,q'} $,
where $q':=q/(q-1)$, and $ \dual{\cdot,\cdot}_{\dot H^{-\alpha,q'}, \dot H^{\alpha,q}} $ denotes
the corresponding duality pairing.
Additionally, the following embedding holds (see, e.g., \cite[Lemma~8.2]{Neerven2012b}):
\begin{equation}
  \label{eq:0}
\mathbb L_{\mathrm{div}}^q \cap \mathbb H^{\alpha,q} \hookrightarrow \dot H^{\alpha,q}, \quad \forall q \in (1,\infty), \, \forall \alpha \in (0,1/q).
\end{equation}
Finally, let \( S(t) \), for \( t \in [0, \infty) \), denote the analytic semigroup generated by \( -A_q \) (see \cite[Theorem~2]{Giga1981}).

\subsection{Stochastic Integrals}
Let \((\Omega, \mathcal{F}, \mathbb{P})\) denote a given complete probability space equipped with a right-continuous filtration \(\mathbb{F} := (\mathcal{F}_t)_{t \geqslant 0}\), on which a sequence of independent Brownian motions \((\beta_n)_{n \in \mathbb{N}}\) is defined. The expectation operator with respect to this probability space will be denoted by \(\mathbb E\).
Consider \( H \) as a separable Hilbert space equipped with the inner product \( (\cdot, \cdot)_H \) and an orthonormal basis \( (h_n)_{n \in \mathbb{N}} \). We define the \( \mathbb{F} \)-adapted \( H \)-cylindrical Brownian motion \( W_H \) such that, for every \( t \geqslant 0 \), \( W_H(t) \) is an element of \( \mathcal{L}(H, L^2(\Omega)) \) given by
\[
  W_H(t)h = \sum_{n \in \mathbb{N}} (h, h_n)_H \beta_n(t), \quad \forall h \in H.
\]
For any UMD Banach space \( E \), let \( \gamma(H, E) \) denote the space of all \( \gamma \)-radonifying operators from \( H \) to \( E \). For any \( p \in [1, \infty] \), we denote by \( L_\mathbb{F}^p(\Omega \times (0, T); \gamma(H, E)) \) the closure in \( L^p(\Omega \times (0, T); \gamma(H, E)) \) of the space
\[
  \text{span}\big\{
    u \mathbbm{1}_{(s,t]\times\mathcal B}: \, 0 \leqslant s < t \leqslant T, 
    \, u \in \gamma(H,E), \, \mathcal B \in \mathcal F_s
  \big\},
\]
where \( \mathbbm{1}_{(s,t]\times\mathcal B} \) is the indicator function of \( (s,t] \times \mathcal B \).
When \( E \) is a Hilbert space, \( \gamma(H,E) \) coincides with the space of
Hilbert-Schmidt operators from \( H \) to \( E \), and Itô's isometry holds:
\[
  \mathbb E \biggl\| \int_0^T g(t) \, \mathrm{d}W_H(t) \biggr\|_{E}^2
  = \mathbb E \int_0^T \| g(t) \|_{\gamma(H,E)}^2 \, \mathrm{d}t,
\]
for all \( g \in L_\mathbb{F}^2(\Omega \times (0, T); \gamma(H, E)) \).
For further details on \(\gamma\)-radonifying operators, we refer to
Chapter 9 of \cite{HytonenWeis2017}, and for stochastic integrals in UMD Banach spaces,
see \cite{Neerven2007}.

\subsection{Global Mild Solutions}
\label{ssec:mild_sol}
We begin by introducing the following hypothesis regarding the diffusion coefficients \( (f_n)_{n\in\mathbb{N}} \) in the model problem (\ref{eq:model}).
\begin{hypothesis}\label{hyp:fn}
Let \( (f_n)_{n\in\mathbb{N}} \) be a sequence of twice continuously
differentiable functions from \( \mathcal{O} \times \mathbb{R}^2 \) to \( \mathbb{R}^2 \).
There exists a constant \( C_F > 0 \) such that the following conditions hold:
\begin{small}
\begin{equation*}
  \begin{cases}
    \sup\limits_{x\in\mathcal O}{}\left(
      \displaystyle\sum_{n\in\mathbb{N}} \bigg[
        |f_n(x,y)|^2 + |\nabla_xf_n(x,y)|^2 
        + |\nabla_x^2f_n(x,y)|^2
      \bigg]
    \right)^{1/2}
    \leqslant C_F \big( 1 + |y| \big), \,\, \forall y \in \mathbb{R}^2, \\
    \sup\limits_{(x,y) \in \mathcal O\times\mathbb R^2}\left(
      \displaystyle\sum_{n\in\mathbb{N}}
      \Big[ \, |\nabla_yf_n(x,y) |^2 + |\nabla_x\nabla_yf_n(x,y)|^2
      + |\nabla_y^2f_n(x,y)|^2  \Big]
    \right)^{1/2}
    \leqslant C_F,
  \end{cases}
\end{equation*}
\end{small}
\end{hypothesis}


Throughout this paper, we consistently assume that Hypothesis \ref{hyp:fn} holds. For any \( u \in \mathbb{L}^2 \), we define \( F(u) \in \gamma(H, \mathbb{L}^2) \) as follows:
\begin{equation}
F(u) := \sum_{n\in\mathbb{N}} h_n \otimes f_n(\cdot,u(\cdot)),
\end{equation}
where the tensor product \( h_n \otimes f_n(\cdot,u(\cdot)) \in \gamma(H, \mathbb{L}^2) \) is defined by
\[
(h_n \otimes f_n(\cdot,u(\cdot)))(w) := (w, h_n)_H f_n(\cdot,u(\cdot)), \quad \forall w \in H.
\]
We present several growth and Lipschitz-continuity properties of the operator \( F \)
in the subsequent lemma. These properties can be verified through straightforward computations,
using the definition of \( \gamma \)-radonifying operators
(see Section 9.1.b in Chapter 9 of \cite{HytonenWeis2017}) and the standard theory
of interpolation spaces (see Lemma~28.1 in \cite{Tartar2007} and Theorem~4.36 in \cite{Lunardi2018}).
\begin{lemma}
  \label{lem:F}
  The following assertions hold:
  \begin{enumerate}
    \item[(i)] There exists a constant \( c > 0 \)
      such that for all \( u, v \in \mathbb{L}^2 \),
      \begin{equation*}
        \left\| F(u) - F(v) \right\|_{\gamma(H, \mathbb{L}^2)} \leqslant c \left\| u - v \right\|_{\mathbb{L}^2}.
      \end{equation*}
    \item[(ii)] There exists a constant \( c > 0 \) such that for all \( u \in \mathbb{L}^2 \),
      \begin{equation*}
        \left\| F(u) \right\|_{\gamma(H, \mathbb{L}^2)} \leqslant c (1 + \left\| u \right\|_{\mathbb{L}^2}).
      \end{equation*}
    \item[(iii)] For any \( 0 \leqslant \alpha \leqslant 1 \),
      there exists a constant \( c > 0 \) such that for all \( u \in \mathbb{H}^{\alpha,2} \), \begin{equation*}
        \left\| F(u) \right\|_{\gamma(H, \mathbb{H}^{\alpha,2})}
        \leqslant c (1 + \left\| u \right\|_{\mathbb{H}^{\alpha,2}}).
      \end{equation*}
    \item[(iv)] For any \( 1 < \varrho < \frac{3}{2} \) and \( \varrho \leqslant \alpha \leqslant 2 \),
      there exists a constant \( c > 0 \) such that for all \( u \in \mathbb{H}^{\alpha,2} \),
      \begin{equation*}
        \left\| F(u) \right\|_{\gamma(H,\mathbb H^{\alpha,2})}
        \leqslant c(1 + \left\| u \right\|_{\mathbb{H}^{\varrho,2}})^{\alpha - 1} \left\| u \right\|_{\mathbb{H}^{\alpha,2}}.
      \end{equation*}
  \end{enumerate}
\end{lemma}

Next, we introduce two convolution operations.
For any $g \in L^p(0,T;\mathbb L_{\mathrm{div}}^q) $ with $p,q \in (1,\infty) $,
the deterministic convolution is defined as
\[
  (S \ast g)(t) := \int_0^t S(t-s)g(s) \, \mathrm{d}s,
  \quad t \in [0,T].
\]
For any stochastic process $g \in L_\mathbb{F}^2(\Omega \times (0, T); \gamma(H, \mathbb{L}_{\text{div}}^q))$
with $ p,q \in [2,\infty) $, the stochastic convolution is defined as:
\[
  (S \diamond g)(t) := \int_0^t S(t-s) g(s) \, \mathrm{d}W_H(s), \quad t \in [0, T].
\]

We now proceed to introduce the concept of a global mild solution as defined
in \cite[pp.~978]{Neerven2008} and \cite[pp.~1396--1397]{Neerven2012b}.
Fix any $\varrho \in (0,\infty)$.
Let $V_\varrho$ denote the space of all $\dot H^{\varrho,2}$-valued,
$\mathbb F$-adapted processes $f: \Omega \times [0,T] \to \dot H^{\varrho,2}$ that
have almost surely continuous paths.
Two processes $ f $ and $ g $ in $ V_\varrho $ are considered identical
if $ f = g $ in $ C([0,T];\dot H^{\varrho,2}) $ holds almost surely.
For any $ p \in [1,\infty) $, let $L_{\mathcal{F}_0}^p(\Omega;\dot H^{\varrho,2})$ denote the space of all
$\dot H^{\varrho,2}$-valued random variables $v$ on the probability space
$(\Omega, \mathcal{F}, \mathbb{P})$ that are strongly $\mathcal{F}_0$-measurable and
satisfy $\mathbb{E}\|v\|_{\dot H^{\varrho,2}}^p < \infty$.
Assume that $y_0 \in L_{\mathcal F_0}^2(\Omega;\mathbb L_{\mathrm{div}}^2) \cap L_{\mathcal{F}_0}^1(\Omega;\dot H^{\varrho,2})$.
The model problem \cref{eq:model} is said to have a global mild solution $y$ in
$V_\varrho$ if, for any $R > 1$, the following equality holds in
$\dot H^{\varrho,2}$, almost surely for all $t \in [0,T]$:
\[
  \begin{aligned}
    y(t \wedge t_{R,\varrho})
    &=  S(t \wedge t_{R,\varrho})y_0 
    - \Big(S \ast  \mathcal{P}\big[(y \cdot \nabla)y \big]\Big)(t \wedge t_{R,\varrho}) \\
    & \quad {} + \Big(S \diamond \mathcal{P}\big[\mathbbm{1}_{[0, t_{R,\varrho}]} F(y)\big]\Big)(t \wedge t_{R,\varrho}),
  \end{aligned}
\]
where $\mathbbm{1}_{[0,t_{R,\varrho}]}$ denotes the indicator function of the time interval $[0,t_{R,\varrho}]$. The stopping time $t_{R,\varrho}$ is defined as
\begin{equation}
  \label{eq:tR}
  t_{R,\varrho} := \inf \left\{ t \in [0, T] : \|y(t)\|_{\dot{H}^{\varrho,2}} \geqslant R \right\},
\end{equation}
with the convention that $t_{R,\varrho} = T$ if this set is empty.

\section{Regularity}
\label{sec:regularity}
The main result of this section is the following proposition.
\begin{proposition}
  \label{prop:regu}
  Let \( y_0 \in L_{\mathcal{F}_0}^2(\Omega;\mathbb L_{\mathrm{div}}^2) \cap L_{\mathcal{F}_0}^1(\Omega; \dot{H}^{\varrho,2}) \)
  with \( \varrho \in (1,3/2) \). Then, the model problem \cref{eq:model} admits a unique global mild solution \( y \) in \( V_\varrho \) satisfying the following properties:
  \begin{enumerate}
    \item[(i)] For some \( c > 0 \) independent of \( R \),
      \begin{equation}
        \label{eq:P-tRrho}
        \mathbb{P}\big( \|y\|_{C([0,T];\dot{H}^{\varrho,2})} \leqslant R\big)
        \geqslant 1 - \frac{c}{\ln(1+R)},
        \quad \forall R > 1.
      \end{equation}
    \item[(ii)] If \( y_0 \in L_{\mathcal{F}_0}^4(\Omega;\mathbb{L}_{\mathrm{div}}^2) \), then
      \begin{equation}
        \label{eq:y-L4-C-L2}
        y \in L^4(\Omega;C([0,T];\mathbb{L}^2)).
      \end{equation}
    \item[(iii)] If \( y_0 \in L_{\mathcal{F}_0}^1(\Omega;\dot{H}^{\alpha,q}) \) with \( q \in (2,\infty) \) and \( \alpha \in (1,1+1/q) \), then for some \( c > 0 \) independent of \( R \),
      \begin{equation}
        \label{eq:P-tRrho-H1q}
        \mathbb{P}\big( \|y\|_{C([0,T];\dot{H}^{\alpha,q})} \leqslant R\big)
        \geqslant 1 - \frac{c}{\ln(1+R)}, \quad \forall R > 1.
      \end{equation}
  \end{enumerate}
\end{proposition}

\begin{remark}
  We summarize some works on the regularity of the two-dimensional SNSEs with no-slip boundary conditions as follows.
  \begin{enumerate}
    \item For additive noise, Menaldi and Sritharan \cite{Menaldi2002} proved global existence of strong solutions with almost surely continuous paths in \( \mathbb{L}^2 \). This result was extended to multiplicative noise by Sritharan and Sundar \cite{Sritharan2006B}.
    \item Under divergence-free diffusion coefficients satisfying homogeneous boundary conditions,
      Glatt-Holtz and Ziane \cite{Glatt-Holtz2009} established global existence of strong pathwise
      solutions with almost surely continuous paths in \( \dot{H}^{1,2} \).
      Building on this, Kukavica and Vicol \cite{Kukavica2014} derived some moment bounds,
      and Breit and Prohl \cite[Section 3]{Breit2023} obtained enhanced regularity properties,
      which play a pivotal role in their numerical analysis.
Additionally, \cite[Theorem~2.4]{Chueshov2010} proved the global existence of a unique weak solution to \cref{eq:model},
 while \cite[Theorem~8.3]{Neerven2012b} established the unique existence of a maximal local mild solution to \cref{eq:model}.
  \end{enumerate}
\end{remark}

The proof of this proposition will be preceded by the introduction of several lemmas that
provide the necessary tools and estimates.
We begin by revisiting some fundamental properties of the Helmholtz projection
operator \( \mathcal{P} \).
\begin{lemma}
  \label{lem:P}
  The Helmholtz projection $\mathcal{P}$ belongs to $\mathcal L(\mathbb{H}^{\alpha,2},\mathbb H^{\alpha,2})$
  for all $\alpha \in [0,2]$, and it induces a bounded linear operator from
  $\mathcal{H}^{-1,2}$ to $\dot{H}^{-1,2}$.
  Moreover, for any $\alpha \in (0,1]$, $\mathcal{P}$ belongs to
  $\mathcal{L}(\mathbb{H}^{\alpha,2}, \dot{H}^{\beta,q})$ for every
  $0 < \beta \leqslant \alpha$ and $2 \leqslant q < \infty$, provided
  that $\frac{2}{1+\beta-\alpha} \geqslant q$ and $ \beta  < 1/q $.
\end{lemma}
\begin{proof}
The results are well-established in the literature.
Specifically, Proposition IV.3.7 in \cite{Boyer2012} demonstrates that
$\mathcal{P}$ is a bounded linear operator from $\mathbb{H}^{1,2}$ to $\mathbb{H}^{1,2}$.
The same argument can be extended to show that $\mathcal{P} \in \mathcal{L}(\mathbb{H}^{\alpha,2}, \mathbb{H}^{\alpha,2})$
for all $\alpha \in [0, 2]$.
The action of $\mathcal{P}$ as a bounded linear operator from
$\mathcal{H}^{-1,2}$ to $\dot{H}^{-1,2}$ is shown in Proposition 9.14 of \cite{Weis2006}.

For the final claim, let $\alpha \in (0,1]$.
Given that $\mathcal P \in \mathcal L(\mathbb H^{\alpha,2},\mathbb H^{\alpha,2})$,
Sobolev's embedding theorem, in conjunction with \cref{eq:0},
directly implies that $\mathcal P$ is also in $\mathcal L(\mathbb H^{\alpha,2},\dot H^{\beta,q})$
under the specified conditions on $\beta$ and $q$. This concludes the proof.
\end{proof}

We summarize several well-established properties of the semigroup \( S(t) \), for \( t \in [0,\infty) \),
in the following lemma. These properties can be found, for instance, in \cite{Geissert2010,Giga1981}.

\begin{lemma}
  \label{lem:S}
  The semigroup \( S(t) \), \( t \in [0, \infty) \), satisfies the following properties:  
  \begin{enumerate}
    \item[(i)] For any $ 0 \leqslant \alpha \leqslant \beta \leqslant 2 $ 
      and $ q \in (1,\infty) $, there exists $ c>0$ such that
      \begin{small}
        \[
          t \left\| \frac{\mathrm{d}}{\mathrm{d}t} S(t) \right\|_{\mathcal{L}(\dot{H}^{\alpha,q}, \dot{H}^{\beta,q})}
          + \left\| S(t) \right\|_{\mathcal{L}(\dot{H}^{\alpha,q}, \dot{H}^{\beta,q})}
          \leqslant ct^{(\alpha-\beta)/2}, \quad \forall t > 0.
        \]
      \end{small}
    \item[(ii)] For any $ g \in L^p(0,T;\mathbb L_{\mathrm{div}}^q) $ with $ p, q \in (1,\infty) $,
      there exists $ c >0 $, independent of $g $ and $ T $, such that
      \begin{small}
        \begin{align*}
          & \left\| \frac{\mathrm{d}}{\mathrm{d}t} S * g \right\|_{L^p(0,T;\mathbb L_{\mathrm{div}}^q)}
          + \left\| S * g \right\|_{L^p(0,T;\dot{H}^{2,q})} 
          + \left\| S * g \right\|_{C([0,T]; (\mathbb L_{\mathrm{div}}^q, \dot{H}^{2,q})_{1-\frac1p,p})} \\
          \leqslant{} & c \left\| g \right\|_{L^p(0,T;\mathbb L_{\mathrm{div}}^q)}.
        \end{align*}
      \end{small}
  \end{enumerate}
\end{lemma}

We are now prepared to prove \cref{prop:regu}.

\medskip\noindent\textbf{Proof of \cref{prop:regu}.}
The proof proceeds in four steps.
Throughout, \(c\) denotes a generic positive constant,
independent of the positive integer \(n\) (introduced in Step 1) and the parameter \(R\),
whose value may change at each occurrence.
For brevity, we will
use the ideal property of \(\gamma\)-radonifying operators
(\cite[Theorem~9.1.10]{HytonenWeis2017}) and
standard properties of real and complex interpolation
spaces (\cite[Propositions 1.3, 1.4, Theorem~4.36]{Lunardi2018},
\cite[Section 11.6]{Yagi2010}) without explicit reference.
For \( \theta \in [0,2] \), \( A_2 \) extends
to a bounded linear operator from \( \dot H^{\theta,2} \) to \( \dot H^{\theta-2,2} \)
via
\[
  \dual{A_2u,v}_{\dot H^{\theta-2},\dot H^{2-\theta,2}}
  := \dual{A_2^{\theta/2}u, A_2^{1-\theta/2}v},
  \quad \forall u \in \dot H^{\theta,2}, \, \forall v \in \dot H^{2-\theta,2}.
\]
These extensions will be used as needed.
The applicability of \cite[Theorem 3.5]{Neerven2012b}
follows from \cite[Theorem~2.5]{Neerven2012b} and
\cite[Theorem 9.17]{Weis2006}.

\textbf{Step 1.} 
Fix \(\varrho_0 \in (\frac13,\frac{1}{2})\).
Let us use the framework from \cite{Neerven2008} to establish the unique existence of
a global mild solution to \cref{eq:model} in \(V_{\varrho_0}\).
For each integer \(n \geqslant 2\), consider the auxiliary problem:
\[
\begin{cases}
\mathrm{d}y^{(n)} = -\left[ A_2y^{(n)}(t) + G_n(y^{(n)}(t)) \right] \mathrm{d}t + \mathcal{P}F(y^{(n)}(t)) \mathrm{d}W_H(t), & t \in [0, T], \\
y^{(n)}(0) = y_0\mathbbm{1}_{\{\norm{y_0}_{\dot H^{\varrho,2}} \leqslant \frac{n}{2}\}},
\end{cases}
\]
where \(\mathbbm{1}_{\{\norm{y_0}_{\dot H^{\varrho,2}} \leqslant \frac{n}{2}\}}\) is the indicator function of the set \(\{\norm{y_0}_{\dot H^{\varrho,2}} \leqslant \frac{n}{2}\}\), and \(G_n\) is defined for \(u \in \dot H^{1,2}\) by:
\[
G_n(u) :=
\begin{cases}
\mathcal{P}\left[ (u \cdot \nabla) u \right], & \| u \|_{\dot{H}^{\varrho_0,2}} \leqslant n, \\
\frac{n^2}{\| u \|_{\dot H^{\varrho_0,2}}^2} \mathcal{P}\left[ (u \cdot \nabla) u \right], & \| u \|_{\dot{H}^{\varrho_0,2}} > n.
\end{cases}
\]
By Lemma~2.2 in \cite{Giga1985B}, \(G_n\) extends to \(\dot H^{\varrho_0,2}\),
and for \(u, v \in \dot H^{\varrho_0,2}\),
\[
\norm{G_n(u)}_{\dot H^{2\varrho_0-2,2}} \leqslant cn\norm{u}_{\dot H^{\varrho_0,2}}, \quad \norm{G_n(u) - G_n(v)}_{\dot H^{2\varrho_0-2,2}} \leqslant cn \norm{u-v}_{\dot H^{\varrho_0,2}}.
\]
From \cref{lem:F} and \cref{lem:P}, for \(u, v \in \dot H^{\varrho_0,2}\),
\[
\norm{\mathcal{P}F(u)}_{\gamma(H,\dot H^{0,2})} \leqslant c(1 + \norm{u}_{\dot H^{\varrho_0,2}}), \quad \norm{\mathcal{P}F(u) - \mathcal{P}F(v)}_{\gamma(H,\dot H^{0,2})} \leqslant c\norm{u-v}_{\dot H^{\varrho_0,2}}.
\]
Let \( p \in (\frac{2}{\varrho_0}, \infty) \) be chosen such that \( (\dot H^{2\varrho_0-2,2}, \dot H^{2\varrho_0,2})_{1-1/p,p} \) and \( (\dot H^{0,2}, \dot H^{2,2})_{1/2-1/p,p} \) are continuously embedded into \( \dot H^{\varrho_0,2} \). By \cref{lem:S} and \cite[Theorem~3.5]{Neerven2012b}, a standard fixed-point argument (see, e.g., \cite[Theorem~6.2]{Neerven2008}) guarantees the existence and uniqueness of a process \( y^{(n)} \in L^p(\Omega;C([0,T];\dot H^{\varrho_0,2})) \) satisfying, almost surely,
\begin{small}
\begin{equation}
  \label{eq:yn-mild}
y^{(n)}(t) = S(t) y_0\mathbbm{1}_{\{\norm{y_0}_{\dot H^{\varrho,2}} \leqslant \frac{n}{2}\}} - \left[ S \ast G_n(y^{(n)}) \right](t) + \left[ S \diamond \mathcal{P}F(y^{(n)}) \right](t) \quad \text{in } \dot H^{\varrho_0,2}
\end{equation}
\end{small}
\hskip -.3em for all \( t \in [0,T] \).
By \cref{lem:S} and \cite[Theorem~3.5]{Neerven2012b}, we deduce that
\begin{align*}
& S(\cdot) y_0\mathbbm{1}_{\{\norm{y_0}_{\dot H^{\varrho,2}} \leqslant \frac{n}{2}\}} + S \diamond \mathcal{P}F(y^{(n)}) \in L_{\mathbb F}^2(\Omega\times(0,T);\dot H^{1,2}), \\
& S \ast G_n(y^{(n)}) \in L_{\mathbb{F}}^2(\Omega \times (0, T); \dot{H}^{2\varrho_0,2}).
\end{align*}
Consequently, \( y^{(n)} \in L_{\mathbb{F}}^2(\Omega \times (0, T); \dot{H}^{2\varrho_0,2}) \), which, by \cite[Lemma~2.2]{Giga1985B}, implies \( G_n(y^{(n)}) \in L_{\mathbb F}^2(\Omega\times(0,T);\dot H^{3\varrho_0-2,2}) \).
Thus, \cref{lem:S}(ii) yields
\[
S \ast G_n(y^{(n)}) \in L_{\mathbb{F}}^2(\Omega \times (0, T); \dot{H}^{3\varrho_0,2}).
\]
Combining these results and noting that \( 3\varrho_0 > 1 \), we conclude
\[
y^{(n)} \in L_{\mathbb{F}}^2(\Omega \times (0, T); \dot{H}^{1, 2}).
\]
Given this regularity and the fact that \( \dual{G_n(u), u} = 0 \) for all \( u \in \dot{H}^{1,2} \), and \( y_0\mathbbm{1}_{\{\norm{y_0}_{\dot{H}^{\varrho,2}} \leqslant \frac{n}{2}\}} \) is uniformly bounded in \( L^2(\Omega; \mathbb L_{\mathrm{div}}^2) \) with respect to \( n \), Itô's formula applied to \( \norm{y^{(n)}}_{\mathbb{L}^2}^2 \) yields
\begin{equation}
  \label{eq:yn-regu}
  \| y^{(n)} \|_{L^2(\Omega \times (0, T); \dot{H}^{1,2})} \quad \text{is uniformly bounded in } n.
\end{equation}
Combined with \cref{lem:F}(iii) and \cref{lem:P}, this implies  
\[
\| \mathcal{P}F(y^{(n)}) \|_{L^2(\Omega \times (0, T); \gamma(H,\dot{H}^{\varrho_0,2}))} \quad \text{is uniformly bounded in } n,
\]
which, by \cite[Theorem 3.5]{Neerven2012b}, implies
\[
\norm{S \diamond \mathcal{P}F(y^{(n)})}_{L^2(\Omega; C([0, T]; \dot{H}^{\varrho_0,2}))} \quad \text{is uniformly bounded in } n.
\]
Using \cref{lem:S}(i) and \(y_0 \in L_{\mathcal{F}_0}^1(\Omega; \dot{H}^{\varrho,2})\), we obtain  
\begin{equation}
\label{eq:zn-regu}
\norm{z^{(n)}}_{L^1(\Omega; C([0, T]; \dot{H}^{\varrho_0,2}))} \quad \text{is uniformly bounded in } n,
\end{equation}
where \(z^{(n)} := S(\cdot)y_0 \mathbbm{1}_{\{\norm{y_0}_{\dot{H}^{\varrho,2}} \leqslant \frac{n}{2}\}} + S \diamond \mathcal{P}F(y^{(n)})\).
Let \(\eta^{(n)} := y^{(n)} - z^{(n)}\). From \cref{eq:yn-mild}, it follows almost surely that
\begin{equation}
\label{eq:etan}
\frac{\mathrm{d}}{\mathrm{d}t}\eta^{(n)}(t) = -A_2\eta^{(n)}(t) - G_n(y^{(n)}(t)) \quad \text{in } \dot{H}^{2\varrho_0-2,2}, \quad \forall t \in [0, T].
\end{equation}
Given \cref{eq:yn-regu}, \cite[Lemma~2.2]{Giga1985B} implies
\( G_n(y^{(n)}) \in L^2(\Omega\times(0,T);\dot H^{\varrho_0-1,2}) \),
leading to \( \eta^{(n)} \in L^2(\Omega\times(0,T);\dot H^{\varrho_0+1,2}) \) by \cref{lem:S}(ii).
This regularity result allows us to deduce from \cref{eq:etan} that, almost surely, for all \(t \in [0, T]\),  
\begin{small}
\begin{align*}
\norm{\eta^{(n)}(t)}_{\dot{H}^{\varrho_0,2}}^2 
&= -2 \int_0^t \dual{A_2\eta^{(n)}(s) + G_n(y^{(n)}(s)), A_2^{\varrho_0} \eta^{(n)}(s)}_{\dot{H}^{\varrho_0-1,2}, \dot{H}^{1-\varrho_0,2}} \, \mathrm{d}s \\
&\leqslant \int_0^t \left( -2\norm{\eta^{(n)}(s)}_{\dot{H}^{1+\varrho_0,2}}^2 + 2 \norm{G_n(y^{(n)}(s))}_{\dot{H}^{\varrho_0-1,2}} \norm{\eta^{(n)}(s)}_{\dot{H}^{1+\varrho_0,2}} \right) \, \mathrm{d}s \\
&\leqslant \int_0^t \norm{G_n(y^{(n)}(s))}_{\dot{H}^{\varrho_0-1,2}}^2 \, \mathrm{d}s.
\end{align*}
\end{small}
\hskip -.3em By \cite[Lemma~2.2]{Giga1985B}, there exists a constant \(c_0 > 0\), independent of \(n\),
such that  $ \norm{G_n(u)}_{\dot{H}^{\varrho_0-1,2}} \leqslant
c_0 \norm{u}_{\dot{H}^{\varrho_0,2}} \norm{u}_{\dot{H}^{1,2}} $
for all $ u \in \dot{H}^{1,2} $. Hence, almost surely, for all \(t \in [0, T]\),  
\[
\norm{\eta^{(n)}(t)}_{\dot{H}^{\varrho_0,2}}^2 \leqslant 2c_0^2 \int_0^t \left( \norm{\eta^{(n)}(s)}_{\dot{H}^{\varrho_0,2}}^2 + \norm{z^{(n)}(s)}_{\dot{H}^{\varrho_0,2}}^2 \right) \norm{y^{(n)}(s)}_{\dot{H}^{1,2}}^2 \, \mathrm{d}s.
\]
Applying Gronwall's inequality, we obtain the following bound, holding almost surely:  
\begin{small}
\[
\norm{\eta^{(n)}}_{C([0, T]; \dot{H}^{\varrho_0,2})}^2 \leqslant 2c_0^2 \exp\left(2c_0^2 \int_0^T \norm{y^{(n)}(s)}_{\dot{H}^{1,2}}^2 \, \mathrm{d}s \right) \int_0^T \norm{y^{(n)}(s)}_{\dot{H}^{1,2}}^2 \norm{z^{(n)}(s)}_{\dot{H}^{\varrho_0,2}}^2 \, \mathrm{d}s.
\]
\end{small}
Consequently, almost surely on the set \(\Omega_n\), defined as  
\begin{small}
\[
\Omega_n := \left\{ \omega \in \Omega : \norm{y^{(n)}}_{L^2(0, T; \dot{H}^{1,2})} \leqslant \min\left\{\frac{n^{1/4}}{2c_0}, \frac{\sqrt{\ln n}}{2c_0}\right\}, \, \norm{z^{(n)}}_{C([0, T]; \dot{H}^{\varrho_0,2})} \leqslant \sqrt{\frac{n}{2}} \right\},
\]
\end{small}
\hskip -.3em we have  
\[
\norm{\eta^{(n)}}_{C([0, T]; \dot{H}^{\varrho_0,2})} \leqslant \frac{n}{2}.
\]
From the decomposition \( y^{(n)} = \eta^{(n)} + z^{(n)} \), it follows that
\begin{equation}
  \label{eq:yn-bound}
  \|y^{(n)}\|_{C([0,T];\dot H^{\varrho_0,2})}
  \leqslant n, \quad\text{almost surely on } \Omega_n.
\end{equation}
In view of \cref{eq:yn-regu,eq:zn-regu}, we deduce that
\begin{align}
  \mathbb P(\Omega_{n}) \geqslant
  1 - \frac{c}{
    \ln n
  } \quad\text{for all } n \geqslant 2.
  \label{eq:omega_n-prob}
\end{align}
For each $ n \geqslant 2 $, define the stopping time
$
  \sigma_n := \inf\{t \in [0,T]: \|y^{(n)}(t)\|_{\dot H^{\varrho_0,2}} \geqslant n\}
  $,
with the convention that $\sigma_n = T$ if the set is empty.
It holds almost surely that $ \sigma_n $ is non-decreasing with respect to $ n $,
and 
$ y^{(m)} = y^{(n)} $ in $ C([0,\sigma_m];\dot H^{\varrho_0,2}) $, for all $ n > m \geqslant 2 $,
as established in \cite[Lemma~8.2]{Neerven2008}.
Consequently, we define the process $y$ such that, almost surely,
$ y := y^{(n)} $ on $ \Omega \times [0,\sigma_n] $
for each $ n \geqslant 2 $.
By \cref{eq:yn-bound,eq:omega_n-prob}, we have
$ \lim_{n \to \infty} \mathbb P(\sigma_n = T) = 1 $,
indicating that $y$ is well-defined on $[0,T]$ almost surely
with continuous paths in $ \dot H^{\varrho_0,2} $.
It follows that $ y $ is a global mild solution of
the model problem \cref{eq:model} in $V_{\varrho_0}$.
Following the argument in \cite[Lemma~8.2]{Neerven2008},
the uniqueness of this global mild solution in \(V_{\varrho_0}\) is confirmed. 
Moreover, \cref{eq:omega_n-prob} implies that
\begin{equation}
  \label{eq:y-varrho0}
  \mathbb P\Big(\norm{y}_{C([0,T];\dot H^{\varrho_0,2})} \leqslant R \Big)
  \geqslant 1 - \frac{c}{\ln(1+R)}
  \quad\text{for all $ R > 1 $}.
\end{equation}



\textbf{Step 2.}
Fix any $ \varrho_0 \in (\max\{\varrho-1,\frac25\},\frac12) $.
Let us proceed to demonstrate that the global mild solution \( y \) of \cref{eq:model}
in \( V_{\varrho_0} \), as constructed in Step 1, is indeed the unique global mild solution of
\cref{eq:model} in \( V_\varrho \). For arbitrary \( R > 1 \),
we decompose \( y(\cdot \wedge t_{R^{1/3},\varrho_0}) \) almost surely as
follows:
\[
y(t \wedge t_{R^{1/3},\varrho_0}) = I_1(t) + I_2(t) + I_3(t)
\quad\text{for all $ t \in [0,T] $},
\]
where $ t_{R^{1/3},\varrho_0} $ is defined by \cref{eq:tR}, and, for any \( t \in [0,T] \), 
\[
\begin{aligned}
I_1(t) &:= S(t \wedge t_{R^{1/3},\varrho_0}) y_0, \\
I_2(t) &:= -\left[ S \ast \mathcal{P}\left(\mathbbm{1}_{[0,t_{R^{1/3},\varrho_0}]} (y \cdot \nabla)y\right) \right](t \wedge t_{R^{1/3},\varrho_0}), \\
I_3(t) &:= \left[S \diamond \mathcal{P}\left(\mathbbm{1}_{[0,t_{R^{1/3},\varrho_0}]} F(y)\right)\right](t \wedge t_{R^{1/3},\varrho_0}).
\end{aligned}
\]
For \( I_1 \), using \cref{lem:S}(i) and \( y_0 \in L_{\mathcal F_0}^1(\Omega;\dot H^{\varrho,2}) \),
we obtain that
\begin{equation}
\label{eq:I1}
\norm{I_1}_{L^1(\Omega;C([0,T];\dot{H}^{\varrho,2}))}
\quad\text{is uniformly bounded in $ R$.}
\end{equation}
For \( I_2 \), by \cite[Lemma~2.2]{Giga1985B}, 
and the definition of \( t_{R^{1/3},\varrho_0} \), it follows that
\[
  \left\| \mathbbm{1}_{[0,t_{R^{1/3},\varrho_0}]} \mathcal P\big[(y \cdot \nabla)y\big] \right\|_{L^{\infty}(\Omega\times(0,T);\dot H^{2\varrho_0-2})} \leqslant c R^{2/3}.
\]
Using \cref{lem:S}(ii) and the embedding
\( (\dot H^{2\varrho_0-2,2}, \dot H^{2\varrho_0,2})_{1-1/p,p} \hookrightarrow \dot H^{4/5,2} \) for
large \( p \) (since $ \varrho_0 > 2/5$), we infer that
\begin{equation}
  \label{eq:I2}
\left\| I_2 \right\|_{L^{1}(\Omega;C([0,T];\dot H^{4/5,2}))} \leqslant c R^{2/3}.
\end{equation}
For \( I_3 \), applying \cref{lem:F}(iii) and \cref{lem:P}, we get
\[
\left\| \mathbbm{1}_{[0,t_{R^{1/3},\varrho_0}]}\mathcal{P}F(y) \right\|_{L_\mathbb{F}^\infty(\Omega\times(0,T); \gamma(H, \dot{H}^{\varrho_0,2}))} \leqslant c R^{1/3}.
\]
Hence, by Theorem 3.5 in \cite{Neerven2012b} and
the embedding \( (\dot{H}^{\varrho_0,2}, \dot{H}^{2+\varrho_0,2})_{1/2-1/p, p}
\hookrightarrow \dot{H}^{\varrho, 2} \) for large \( p \) (since \( \varrho_0 > \varrho-1\)), we conclude
\begin{equation}
\label{eq:I3}
\left\| I_3 \right\|_{L^1(\Omega;C([0,T]; \dot{H}^{\varrho, 2}))} \leqslant c R^{1/3}.
\end{equation}
Combining \cref{eq:I1,eq:I2,eq:I3}, we obtain  
\begin{equation*}  
\lVert y(\cdot \wedge t_{R^{1/3},\varrho_0}) \rVert_{L^1(\Omega; C([0,T]; \dot{H}^{4/5,2}))} \leqslant c R^{2/3}.  
\end{equation*}  
This inequality, together with \eqref{eq:y-varrho0}, yields  
\[  
\mathbb P\big( \norm{y}_{C([0,T];\dot H^{4/5,2})} \leqslant R\big)  
\geqslant 1 - \frac{c}{\ln(1+R)}.  
\]  
Based on this inequality and the estimate from \cite[Lemma~2.2]{Giga1985B},  
\[  
\norm{\mathcal P\big[(u\cdot\nabla) u\big]}_{\dot H^{-2/5,2}}  
\leqslant c \norm{u}_{\dot H^{4/5,2}}^2,  
\quad \forall u \in \dot H^{1,2},  
\]  
we reiterate the previous argument, replacing \( t_{R^{1/3},\varrho_0} \) with  
\( t_{R^{1/3},4/5} \), to establish the desired probability inequality \eqref{eq:P-tRrho}.  
Furthermore, \eqref{eq:P-tRrho} confirms that \( y \) is a global mild solution  
of the model problem \eqref{eq:model} in \( V_\varrho \).  
Uniqueness follows since the global mild solution in \( V_{\varrho} \)  
is also a global mild solution in \( V_{\varrho_0} \), which is unique by Step 1.

\textbf{Step 3.}
We now prove \cref{eq:y-L4-C-L2} under the additional assumption \( y_0 \in L_{\mathcal{F}_0}^4(\Omega;\mathbb{L}_{\mathrm{div}}^2) \). For each \( n \in \mathbb{N}_{>0} \), we have
\begin{small}
\[
\mathrm{d}y(t\wedge t_{n,\varrho}) = \mathbbm{1}_{[0,t_{n,\varrho}]}(t)\left\{\big(-A_2y(t) - \mathcal{P}[(y(t)\cdot\nabla)y(t)] \big) \mathrm{d}t + \mathcal{P}F(y(t)) \mathrm{d}W_H(t) \right\}, \quad t \in [0,T],
\]
\end{small}
\hskip -.3em where \( t_{n,\varrho} \) is defined in \cref{eq:tR}. 
By \cref{lem:F}(ii) and the identity \( \dual{\mathcal{P}[(u\cdot\nabla)u],u} = 0 \)
for \( u \in \dot{H}^{\varrho,2} \),
an application of Itô's formula and the Burkholder-Davis-Gundy inequality shows that
$
\norm{y(\cdot\wedge t_{n,\varrho})}_{L^4(\Omega;C([0,T];\mathbb{L}^2))}
$ is uniformly bounded in \( n \).
The monotone convergence theorem and Fatou's lemma then imply
\[
\mathbb{E}\left[\norm{y}_{C([0,T];\mathbb{L}^2)}^4\right] \leqslant \liminf_{n \to \infty} \mathbb{E}\left[\norm{y(\cdot\wedge t_{n,\varrho})}_{C([0,T];\mathbb{L}^2)}^4\right] < \infty,
\]
establishing the regularity claim in \cref{eq:y-L4-C-L2}.

\textbf{Step 4.}
We now establish \cref{eq:P-tRrho-H1q} under the additional condition
\( y_0 \in L_{\mathcal{F}_0}^1(\Omega;\dot{H}^{\alpha,q}) \)
with \( q \in (2,\infty) \) and \( \alpha \in (1,1+1/q) \).
Fix \( R > 1 \), and let \( t_{R^{1/3},\varrho} \) be defined as in \cref{eq:tR}.
Given that \( y_0 \in L_{\mathcal{F}_0}^1(\Omega; \dot{H}^{\alpha,q}) \),
Lemma \ref{lem:S}(i) implies  
\begin{equation}  
  \label{eq:9}  
  S(\cdot)y_0 \in L^1(\Omega; C([0,T]; \dot{H}^{\alpha,q})).  
\end{equation}  
By Hölder's inequality and the embedding
$ \dot H^{\varrho,2}\hookrightarrow\mathbb L^\infty$,
we have
\[
  \norm{\mathcal P\big[(y \cdot \nabla)y\big]}_{L^\infty(0,T; \dot H^{0,2}))}
  \leqslant cR^{2/3},
\]
almost surely on the event \( \{t_{R^{1/3},\varrho}=T\} \).
Consequently, using \cref{lem:S}(ii) with
the embedding $ (\dot H^{0,2},\dot H^{2,2})_{1-1/p,p} \hookrightarrow \dot H^{\alpha,q} $
for sufficiently large $p \in (2,\infty)$, we obtain  
\begin{align}  
  \norm{S \ast \big[\mathcal{P}\big[(y \cdot \nabla)y\big]\big]}_{  
    L^1(\{t_{R^{1/3},\varrho}=T\}; C([0,T]; \dot{H}^{\alpha,q}))  
  }  
  \leqslant cR^{2/3}.  
  \label{eq:10}  
\end{align}  
From \cref{lem:F}(iii) we have
$
  \norm{\mathbbm{1}_{[0,t_{R^{1/3},\varrho}]}F(y)}_{\mathbb L_{\mathbb F}^\infty(\Omega\times(0,T);\gamma(H,\mathbb H^{1,2}))}
  \leqslant cR^{1/3}
  $,
which implies, by \cref{lem:P}, for any $ 0 < \beta < 1/q $.
\begin{align*}
  \norm{\mathbbm{1}_{[0,t_{R^{1/3},\varrho}]}\mathcal PF(y)}_{\mathbb L_{\mathbb F}^\infty(\Omega\times(0,T);\gamma(H,\dot H^{\beta,q}))}
  \leqslant cR^{1/3}.
\end{align*}
Using \cite[Theorem~3.5]{Neerven2012b} and
the embedding $ (\dot H^{\beta,q},\dot H^{2+\beta,q})_{1/2-1/p,p}
\hookrightarrow \dot H^{\alpha,q} $
for sufficiently large $ \beta \in (0,1/q) $ and $ p \in (2,\infty) $,
we get
\begin{align}  
  \left\lVert S \diamond \left(\mathbbm{1}_{[0,t_{R^{1/3},\varrho}]} \mathcal{P}F(y)\right)\right\rVert_{  
    L^1(\Omega; C([0,T]; \dot{H}^{\alpha,q}))  
  } \leqslant cR^{1/3}.  
  \label{eq:11}  
\end{align}  
From the identity  
\[  
  y = S(\cdot)y_0 - S \ast \mathcal{P}\big[(y \cdot \nabla)y\big]  
  + S \diamond \big[\mathbbm{1}_{[0,t_{R^{1/3},\varrho}]} \mathcal{P}F(y)\big],  
\]  
valid almost surely on \( \{t_{R^{1/3},\varrho}=T\} \) in \( C([0,T];\dot H^{\varrho,2}) \),
and by combining \cref{eq:9,eq:10,eq:11}, we derive  
\[  
  \norm{y}_{L^1(\{t_{R^{1/3},\varrho}=T\}; C([0,T]; \dot{H}^{\alpha,q}))}  
  \leqslant cR^{2/3}.  
\]  
This, together with \cref{eq:P-tRrho}, establishes \cref{eq:P-tRrho-H1q}, thereby completing the proof of \cref{prop:regu}.

\hfill$\blacksquare$

\section{Spatial Semidiscretization}
\label{sec:spatial_discretization}
Let \( \mathcal{K}_h \) be a conforming and quasi-uniform triangulation of
the domain \( \mathcal{O} \), with each $ K \in \mathcal K_h $ containing
at least one interior vertex to satisfy the conditions in \cite{Brezzi1991}
and \cite{Girault2003}.
Let \( h \) represents the maximum diameter of
the elements in \( \mathcal{K}_h \). The closure of the union of all elements
in \( \mathcal{K}_h \) is denoted by \( \mathcal{O}_h \). Let \( \mathbb{L}_h \) and
\( M_h \) be constructed on \( \mathcal{O}_h \) using the well-known
\( P_3 / P_2 \) Taylor-Hood finite element method \cite{Stenberg1990}.
Specifically,
\begin{align*}
  \mathbb{L}_h &:= \big\{u_h \in C(\overline{\mathcal{O}_h}; \mathbb{R}^2): u_h|_K \in P_3(K; \mathbb{R}^2) \text{ for each } K \in \mathcal{K}_h \text{ and } u_h = 0 \text{ on } \partial\mathcal{O}_h \big\}, \\
  M_h &:= \big\{\phi_h \in C(\overline{\mathcal{O}_h}; \mathbb{R}):
  \phi_h|_K \in P_2(K) \text{ for each } K \in \mathcal{K}_h \big\},
  \end{align*}
where \( P_3(K; \mathbb{R}^2) \) denotes the space of all vector-valued polynomials
of degree at most 3 on \( K \), and \( P_2(K) \) denotes the space of all polynomials
of degree at most 2 on \( K \). Functions in \( \mathbb{L}_h \) and \( M_h \) are extended by zero
to \( \mathcal{O} \setminus \mathcal{O}_h \).
We define \( \mathbb{L}_{h,\text{div}} \) as the subset of \( \mathbb{L}_h \) consisting of all \( u_h \in \mathbb{L}_h \) such that
\[
\langle \nabla \cdot u_h, \phi_h \rangle = 0, \quad \forall \phi_h \in M_h.
\]
Let \( \mathcal{P}_h \) denote the \( L^2(\mathcal{O}; \mathbb{R}^2) \)-orthogonal projection operator onto \( \mathbb{L}_{h,\text{div}} \). The discrete negative Stokes operator \( A_h: \mathbb{L}_{h,\text{div}} \to \mathbb{L}_{h,\text{div}} \) is defined by
\[
  \dual{A_h u_h, v_h} = \dual{\nabla u_h, \nabla v_h}, \quad \forall u_h, v_h \in \mathbb{L}_{h,\text{div}}.
\]
Here, we recall that $ \dual{\cdot,\cdot} $ denotes the inner product in
\(L^2(\mathcal O)\), \( L^2(\mathcal{O}; \mathbb{R}^2) \), or \( L^2(\mathcal{O}; \mathbb{R}^{2 \times 2}) \),
as appropriate for the context.
It is standard that, for any $ \theta \in [0,1] $, there
exists $ c >0 $, independent of $h $, such that
\begin{equation}
  \label{eq:inverse}
  \norm{A_h^\theta}_{\mathcal L(\mathbb L^2,\mathbb L^2)}
  \leqslant ch^{-2\theta}.
\end{equation}
For any \( \theta \in \mathbb{R} \), we define the space \( \dot{H}_h^{\theta,2} \) as \( \mathbb{L}_{h,\text{div}} \) equipped with the norm
\[
  \| u_h \|_{\dot{H}_h^{\theta,2}} := \| A_h^{\theta/2} u_h \|_{L^2(\mathcal{O}_h)}, \quad \forall u_h \in \mathbb{L}_{h,\text{div}}.
\]
By Theorem~16.1 in \cite{Yagi2010}, the complex interpolation
theory \cite[Theorem~2.6]{Lunardi2018} applies to the spaces \( \dot{H}_h^{\theta,2} \),
\( \theta \in \mathbb{R} \), with constants independent of \( h \).
As an application, for any \( \theta \in [0,1] \), there exists \( c > 0 \), independent of \( h \), such that
\begin{equation}
  \label{eq:1}
  \| u_h \|_{\mathbb{H}^{\theta,2}} \leqslant c \| u_h \|_{\dot{H}_h^{\theta,2}}, \quad \forall u_h \in \mathbb{L}_{h,\text{div}}.
\end{equation}


We study the following spatial semidiscretization:
\begin{equation}
  \label{eq:yh}
  \begin{cases}
    \mathrm{d}y_h(t) = -\big[
      A_h y_h(t) + \mathcal{P}_h G(y_h(t))
    \big] \mathrm{d}t
    + \mathcal{P}_h F(y_h(t)) \mathrm{d}W_H(t), \quad t \in [0, T], \\
    y_h(0) = \mathcal{P}_h y_0,
  \end{cases}
\end{equation}
where the operator \( G \) is defined as
\begin{equation}
  \label{eq:G-def}
  G(u) := (u \cdot \nabla) u + \frac{1}{2} (\nabla \cdot u) u,
  \quad \forall u \in \mathbb H^{1,2}.
\end{equation}
By the classical theory of finite-dimensional stochastic differential equations (see, e.g., Theorem~3.27 of \cite{Pardoux2014}),
for any initial condition $ y_0 \in L_{\mathcal F_0}^2(\Omega;\mathbb L_{\mathrm{div}}^2) $,
\eqref{eq:yh} admits a unique strong solution.

The main results of this section are summarized in the following theorems.

\begin{theorem}
  \label{thm:yh-regu}
  Suppose \( y_0 \in L_{\mathcal{F}_0}^4(\Omega;\mathbb{L}_{\mathrm{div}}^2) \cap L_{\mathcal{F}_0}^1(\Omega; \dot{H}^{\varrho,2}) \) with \( \varrho \in (1, \frac{3}{2}) \). Let \( y_h \) be the strong solution to the spatial semidiscretization \eqref{eq:yh}. There exists a constant \( c > 0 \), independent of \( h \) and \( R_h \), such that for all \( 0 < h < 1 \) and \( 1 < R_h < \infty \),
  \begin{equation}
    \label{eq:yh-regu}
    \mathbb{P}\left(\norm{y_h}_{C([0,T];\dot{H}_h^{\varrho,2})} \leqslant R_h\right)
    \geqslant 1 - \frac{c}{\ln(1+R_h)}.
  \end{equation}
\end{theorem}

\begin{theorem}
  \label{thm:yR-yhR}
  Assume \( y_0 \in L_{\mathcal F_0}^p(\Omega;\dot{H}^{\varrho,2}) \)
  with $ p \in [4,\infty) $ and \( \varrho \in (1, \frac{3}{2}) \).
  Let \( y \) be the global mild solution in \( V_\varrho \) to \eqref{eq:model}, and let \( y_h \) be the strong solution of \eqref{eq:yh}. Define \( t_{R,\varrho} \) as in \eqref{eq:tR}.
  Then, there exists a constant \( c > 0 \), independent of \( h \) and \( R \),
  such that for all \( 0 < h < 1/2 \) and \( R > 1 \),
  \begin{equation}
    \label{eq:yR-yhR}
    \|y - y_h\|_{L^p(\{t_{R,\varrho}=T\}; C([0,T]; \mathbb{L}^2))}
    \leqslant ch^{\varrho}\ln\frac{1}{h}\exp(cR^2).
  \end{equation}
\end{theorem}

\begin{theorem}
  \label{thm:y-yh-global}
  Under the assumptions of Theorem \ref{thm:yR-yhR},
  there exists \( h_0 \in (0, \exp(-3)) \) such that for all \( 0 < h < h_0 \),
  \begin{equation}
    \label{eq:y-yh-global}
    \norm{y - y_h}_{L^2(\Omega; C([0,T]; \mathbb{L}^2))}
    \leqslant \frac{c}{\Big(\ln\big(\ln\frac{1}{h}\big)\Big)^{1/4}},
  \end{equation}
  where \( c > 0 \) is a constant independent of \( h \).
\end{theorem}

To prove these theorems, we introduce several auxiliary lemmas. For convenience, throughout this section,
\( c \) denotes a generic positive constant whose value is independent of \( h \),
\( R \) in the stopping time \( t_{R,\varrho} \),
and $ R_h $ in \cref{eq:yh-regu}, and may vary from one occurrence to another.

Firstly, we summarize some standard results about the semigroup generated by
$ -A_h $ in $ \dot H_h^{0,2} $, defined by $ S_h(t) := \exp(-tA_h) $ for all $ t \in [0,\infty) $.
The deterministic convolution of $ S_h $ with $ g_h \in L^1(0,T;\dot H_h^{0,2}) $ is
given by
\[
  (S_h \ast g_h)(t) := \int_0^t S_h(t-s)g_h(s) \, \mathrm{d}s,
  \quad t \in [0,T].
\]
The stochastic convolution of $ S_h $ with
$ g_h \in L_\mathbb{F}^2(\Omega \times (0,T); \gamma(H,\dot H_h^{0,2})) $
is defined as
\[
  (S_h \diamond g_h)(t) := \int_0^t S_h(t-s)g_h(s) \, \mathrm{d}W_H(s),
  \quad t \in [0,T].
\]
\begin{lemma}
  \label{lem:Sh}
  The following properties of the semigroup $ S_h $ hold:
  \begin{enumerate}
    \item[(i)] For any \( 0 \leqslant \alpha \leqslant \beta \leqslant 2 \) and \( t \in (0, \infty) \), the following estimates hold:
      \begin{align*}
        & \norm{S_h(t)}_{\mathcal L(\dot H_h^{\alpha,2}, \dot H_h^{\beta,2})}
        \leqslant ct^{(\alpha-\beta)/2}, \\
        & \norm{I-S_h(t)}_{\mathcal L(\dot H_h^{\beta,2}, \dot H_h^{\alpha,2})}
        \leqslant ct^{(\beta-\alpha)/2}.
      \end{align*}
    \item[(ii)] For any \( g_h \in L^p(0,T; \dot{H}_h^{0,2}) \) with \( p \in (1, \infty) \), the following estimate holds:
      \begin{align*}
        & \Big\lVert \frac{\mathrm{d}}{\mathrm{d}t} S_h\ast g_h \Big\rVert_{L^p(0,T;\dot H_h^{0,2})}
        + \lVert S_h \ast g_h \rVert_{L^p(0,T;\dot H_h^{2,2})} \\
        & \quad {} + \| S_h \ast  g_h \|_{C([0,T]; (\dot{H}_h^{0,2}, \dot{H}_h^{2,2})_{1-1/p,p})} 
         \leqslant c \| g_h \|_{L^p(0,T; \dot{H}_h^{0,2})}.
      \end{align*}
    \item[(iii)] For any \( g_h \in L_\mathbb{F}^p(\Omega \times (0,T); \gamma(H,\dot{H}_h^{0,2})) \) with \( p \in [2, \infty) \), the following estimate holds:
      \[
        \| S_h \diamond g_h \|_{L^p(\Omega; C([0,T]; (\dot{H}_h^{0,2}, \dot{H}_h^{2,2})_{1/2-1/p,p}))} \leqslant c \| g_h \|_{L^p(\Omega \times (0,T); \gamma(H,\dot{H}_h^{0,2}))}.
      \]
  \end{enumerate}
\end{lemma}
\begin{proof}
  The operator \( A_h \) is symmetric and positive-definite on \( \dot{H}_h^{0,2} \), satisfying
  \[
  \langle A_h v_h, v_h \rangle = \|\nabla v_h\|_{\mathbb{L}^2}^2 \quad \forall v_h \in \dot{H}_h^{0,2}.
  \]
  This implies that the smallest eigenvalue of \( A_h \) is uniformly bounded away from zero with respect
  to \( h \). As a result, part (i) can be established
  through direct computation involving the spectral decomposition of \(A_h\); for related techniques, see Chapter 3 of \cite{Thomee2006}.
  For part (ii), we refer the reader to Theorems 4.24 and 4.29, and Corollary 1.14 in \cite{Lunardi2018}.
  Regarding part (iii), the reader may consult, for example, Theorems~2.5 and 3.5 in \cite{Neerven2012b} and Theorem 10.2.24 in \cite{HytonenWeis2017}; see also \cite[Theorem~6.20]{Prato2014}.
\end{proof}

Secondly, we present some results on the Helmholtz projection operator $ \mathcal P $
and the discrete Helmholtz projection operator \( \mathcal{P}_h \).

\begin{lemma}
  \label{lem:Ph}
  The following assertions hold:
  \begin{enumerate}
    \item[(i)] For any \( \alpha \in [0, 2] \), it holds that
      $ \norm{I - \mathcal{P}_h}_{\mathcal{L}(\dot{H}^{\alpha, 2}, \mathbb{L}^2)} \leqslant c h^\alpha $.
    \item[(ii)] For all \( \alpha \in [0, 2] \), \( \mathcal{P}_h \) is uniformly bounded  in \( \mathcal{L}(\dot{H}^{\alpha, 2}, \dot{H}_h^{\alpha, 2}) \) with respect to \( h \).
    \item[(iii)] For all \( u_h \in \mathbb L_{h,\mathrm{div}} \) and \( v \in \mathbb{H}^{\alpha,2} \) with
      \( \alpha \in [0, 2] \), the following duality estimate holds:
      \begin{align}
        \langle u_h, v - \mathcal{P} v \rangle
        \leqslant c h^{\alpha+1} \, \|\nabla u_h\|_{\mathbb{L}^2} \, \|v\|_{\mathbb{H}^{\alpha,2}}.
        \label{eq:918} 
      \end{align}
    \item[(iv)] For any \( \alpha \in [0, 2] \), the following inequality holds:
      \begin{equation}
        \label{eq:919}
        \norm{\mathcal P_h(I-\mathcal P)}_{\mathcal L(\mathbb H^{\alpha,2},\mathbb L^2)}
        \leqslant ch^{\alpha}.
      \end{equation}
    \item[(v)] For any \( \alpha \in [0, 1] \), the following approximation property holds:
      \begin{equation}
        \label{eq:1000}
        \norm{I-\mathcal P}_{\mathcal L(\dot H_h^{\alpha,2},\mathbb L^2)}
        \leqslant ch^\alpha.
      \end{equation}
    \item[(vi)] For any \( \alpha \in [0, 1/2) \),
      \( \mathcal{P}_h \) is uniformly bounded in \( \mathcal{L}(\mathbb{H}^{\alpha,2}, \dot{H}_h^{\alpha,2}) \) with respect to \( h \). 
  \end{enumerate}
\end{lemma}
\begin{proof}
Assertions (i) and (ii) are well-known results, and their proofs are omitted. We now prove assertion (iii). Fix \( u_h \in \mathbb{L}_{h,\mathrm{div}} \) and \( v \in \mathbb{H}^{\alpha,2} \) with \( \alpha \in [0, 2] \). By Lemma~III.1.2 in \cite{Galdi2011} and Theorem~III.4.3 in \cite{Boyer2012}, there exists \( \varphi \in H^{\alpha+1,2} \) such that \( v - \mathcal{P}v = \nabla \varphi \), with  
\begin{equation}
\|\varphi\|_{H^{\alpha+1,2}} \leqslant c \|v\|_{\mathbb{H}^{\alpha,2}}.  
\label{eq:921}  
\end{equation}
The proof proceeds as follows:  
\begin{align*}  
\langle u_h, v - \mathcal{P}v \rangle  
&= \langle u_h, \nabla \varphi \rangle  
= -\langle \nabla \cdot u_h, \varphi \rangle \quad \text{(integration by parts)} \\  
&= -\langle \nabla \cdot u_h, \varphi - \Pi_h \varphi \rangle \quad \text{(since } u_h \in \mathbb{L}_{h,\text{div}} \text{)} \\  
&\leqslant \|\nabla \cdot u_h\|_{L^2} \, \|\varphi - \Pi_h \varphi\|_{L^2} \\  
&\leqslant c\|\nabla u_h\|_{\mathbb{L}^2} \, \|\varphi - \Pi_h \varphi\|_{L^2},  
\end{align*}  
where \( \Pi_h \varphi \) is the \( L^2 \)-orthogonal projection of \( \varphi \)
onto \( M_h \). Combining \cref{eq:921} with the standard approximation estimate  
\[  
\|\varphi - \Pi_h \varphi\|_{L^2} \leqslant ch^{\alpha+1} \, \|\varphi\|_{H^{\alpha+1,2}},  
\]  
yields the desired estimate \cref{eq:918}.

For assertion (iv), applying \cref{eq:918} yields for \( \alpha \in [0,2] \) and \( u \in \mathbb H^{\alpha,2} \):
\begin{align*}  
  \norm{\mathcal P_h(u-\mathcal Pu)}_{\mathbb L^2}^2  
  &= \dual{\mathcal P_h(u-\mathcal Pu), u - \mathcal Pu} \\  
  &\leqslant ch^{\alpha+1} \norm{\nabla\mathcal P_h(u-\mathcal Pu)}_{\mathbb L^2} \norm{u}_{\mathbb H^{\alpha,2}} \\  
  &\leqslant ch^{\alpha} \norm{\mathcal P_h(u-\mathcal Pu)}_{\mathbb L^2} \norm{u}_{\mathbb H^{\alpha,2}},  
  \quad\text{(by \cref{eq:inverse})}
\end{align*}  
which implies \cref{eq:919}.

We now proceed to prove assertion (v). For any \( u_h \in \mathbb L_{h,\mathrm{div}} \),
there exists \( \varphi \in H^{1,2} \) such that \( u_h - \mathcal Pu_h = \nabla\varphi \) with \( \norm{\varphi}_{H^{1,2}} \leqslant c \norm{u_h-\mathcal Pu_h}_{\mathbb L^2} \). Thus,
\begin{align*}
  \norm{u_h-\mathcal Pu_h}_{\mathbb L^2}^2
  &= \dual{u_h - \mathcal Pu_h, \nabla\varphi}
  = -\dual{\nabla\cdot u_h, \varphi} \\
  & \leqslant ch\norm{\nabla u_h}_{\mathbb L^2} \norm{\varphi}_{H^{1,2}}
  \leqslant ch\norm{u_h}_{\dot H_h^{1,2}} \norm{u_h-\mathcal Pu_h}_{\mathbb L^2}.
\end{align*}
By invoking \cref{eq:inverse}, it follows that \cref{eq:1000} holds
for all \( \alpha \in [0,1] \). Here, we have used techniques analogous
to those applied in proving assertion (iii).

Finally, we prove assertion (vi). Fix \( \alpha \in [0,1/2) \).
Using \cref{eq:919} and the inverse estimate \cref{eq:inverse}, we infer that
\[
  \norm{\mathcal P_hu}_{\dot H_h^{\alpha,2}}
  \leqslant \norm{\mathcal P_h\mathcal Pu}_{\dot H_h^{\alpha,2}}
  + c\norm{u}_{\mathbb H^{\alpha,2}},
  \quad \forall u \in \mathbb H^{\alpha,2}.
\]
By \cref{lem:P} and assertion (ii) of this lemma,
\( \norm{\mathcal{P}_h}_{\mathcal{L}(\mathbb{H}^{\alpha,2},\dot{H}_h^{\alpha,2})} \)
is uniformly bounded in \( h \). This completes the proof of assertion (vi) and the lemma.  
\end{proof}

Thirdly, for any \(\alpha \in [0,1]\), we define the operator \(\mathscr{P}_h: \dot{H}^{-\alpha,2} \to \mathbb{L}_{h,\mathrm{div}}\) by
\[
\mathscr{P}_h := \mathcal{P}_h(-\Delta_2) A_2^{-1},
\]
where:
\begin{itemize}
\item \(A_2\) extends to a bounded linear operator from \(\dot H^{2-\alpha,2}\) to \(\dot H^{-\alpha,2}\) via the duality pairing:
  \[
  \dual{A_2 u, v}_{\dot H^{-\alpha,2}, \dot H^{\alpha,2}}
  := \dual{A_2^{1-\alpha/2}u, A_2^{\alpha/2}v},
  \quad\forall u \in \dot H^{2-\alpha,2}, \, \forall v \in \dot H^{\alpha,2}.
  \]
  This extension admits a bounded inverse $ A_2^{-1} $ from $ \dot H^{-\alpha,2} $ to $ \dot H^{2-\alpha,2} $.
\item \(-\Delta_2\) extends to a bounded linear operator from \(\mathcal H^{2-\alpha,2}\) to \(\mathcal H^{-\alpha,2}\) via the duality pairing:
  \[
  \dual{-\Delta_2 u, v}_{\mathcal H^{-\alpha,2}, \mathcal H^{\alpha,2}}
  := \dual{(-\Delta_2)^{1-\alpha/2}u, (-\Delta_2)^{\alpha/2}v},
  \quad\forall u \in \mathcal H^{2-\alpha,2}, \, \forall v \in \mathcal H^{\alpha,2}.
  \]
\item \(\mathcal{P}_h\) extends to \(\mathcal{H}^{-\alpha,2} \to \mathbb{L}_{h,\mathrm{div}}\) by
  \[
  \langle \mathcal{P}_h f, u_h \rangle = \langle f, u_h \rangle_{\mathcal{H}^{-\alpha,2}, \mathcal{H}^{\alpha,2}}, \quad 
  \forall f \in \mathcal{H}^{-\alpha,2},\ u_h \in \mathbb{L}_{h,\mathrm{div}}.
  \]
\end{itemize}
This operator satisfies the following properties.
\begin{lemma}
  \label{lem:scrPh}
  Let $ \alpha \in [0,2] $. The following assertions hold:
  \begin{enumerate}
    \item[(i)] For any $ u_h \in \mathbb L_{h,\mathrm{div}} $ and $ v \in \mathbb H^{\alpha,2} $, we have  
      \begin{align*}
        \dual{u_h, \, (\mathscr P_h\mathcal P - I)v}
        \leqslant ch^{\alpha+1}\norm{\nabla u_h}_{\mathbb L^2}
        \norm{v}_{\mathbb H^{\alpha,2}}.
      \end{align*}
    \item[(ii)] For any $ u \in \mathbb L_{\mathrm{div}}^2 \cap \mathbb H^{\alpha,2} $,  
      it holds that $ \norm{(\mathcal P_h-\mathscr P_h)u}_{\mathbb L^2} \leqslant ch^{\alpha} \norm{u}_{\mathbb H^{\alpha,2}} $.
    \item[(iii)] The operator norm satisfies  
      $ \norm{I - \mathscr P_h}_{\mathcal L(\dot H^{\alpha,2},\mathbb L^2)} \leqslant ch^\alpha $.
  \end{enumerate}
\end{lemma}
\begin{proof}
To begin with assertion (i), consider any $u_h \in \mathbb{L}_{h,\mathrm{div}}$ and $v \in \mathbb{H}^{\alpha,2}$. By definition,
\begin{align*}
  \langle u_h, (\mathscr{P}_h\mathcal{P} - I)v \rangle 
  = \left\langle u_h, \, (\mathcal{P} - I) \big(\Delta_2 A_2^{-1}\mathcal{P}v + v\big) \right\rangle.
\end{align*}
Applying Lemma \ref{lem:Ph}(iii) yields
\begin{align*}
    \langle u_h, (\mathscr{P}_h\mathcal{P} - I)v \rangle 
    &\leqslant c h^{\alpha+1} \norm{\nabla u_h}_{\mathbb{L}^2} \left( \norm{\Delta_2 A_2^{-1}\mathcal{P}v}_{\mathbb{H}^{\alpha,2}} + \norm{v}_{\mathbb{H}^{\alpha,2}} \right).
\end{align*}
Using Theorem IV.5.8 from \cite{Boyer2012} and interpolation arguments, we deduce
\[
    \norm{\Delta_2 A_2^{-1}\mathcal{P}v}_{\mathbb{H}^{\alpha,2}} \leqslant c \norm{v}_{\mathbb{H}^{\alpha,2}},
\]
which, combined with the previous inequality, completes the proof for assertion (i).

For assertion (ii), let $u \in \mathbb{L}_{\mathrm{div}}^2 \cap \mathbb{H}^{\alpha,2}$. Then,
\begin{align*}
    \norm{(\mathcal{P}_h - \mathscr{P}_h)u}_{\mathbb{L}^2}^2 
    &= \langle (\mathcal{P}_h - \mathscr{P}_h)u, (\mathcal{P}_h - \mathscr{P}_h)u \rangle \\
    &= \langle (\mathcal{P}_h - \mathscr{P}_h)u, (I - \mathscr{P}_h\mathcal{P})u \rangle.
\end{align*}
By assertion (i) and the inverse estimate \cref{eq:inverse}, we obtain
\begin{align*}
    \norm{(\mathcal{P}_h - \mathscr{P}_h)u}_{\mathbb{L}^2}^2 
    &\leqslant c h^{\alpha} \norm{(\mathcal{P}_h - \mathscr{P}_h)u}_{\mathbb{L}^2} \norm{u}_{\mathbb{H}^{\alpha,2}},
\end{align*}
which implies the desired conclusion for assertion (ii).

Finally, for assertion (iii), given the identity
\[
    I - \mathscr{P}_h = (I - \mathcal{P}_h) + (\mathcal{P}_h - \mathscr{P}_h),
\]
the assertion follows by combining Lemma \ref{lem:Ph}(i) with assertion (ii).
This concludes the proof.
\end{proof}

Fourthly, we present some auxiliary estimates.

\begin{lemma}
  \label{lem:y-phy-xih}
  Assume that the hypotheses of Theorem \ref{thm:yR-yhR} are fulfilled. Let
  \[
    \xi_h := S_h \ast (\mathscr P_h A_2 y - A_h \mathscr P_h y).
  \]
  Then, for any $ 0 < h < 1/2 $ and \( R > 1 \), the following inequalities hold almost surely on
  the set $ \{\norm{y_0}_{\dot H^{\varrho,2}}\leqslant R\}$:
  \begin{align}
        & \norm{y - \mathscr{P}_h y - \xi_h}_{C([0, t_{R,\varrho}]; \mathbb{L}^2)} \leqslant c h^\varrho
        \left(\ln\frac{1}{h}\right) R,
        \label{eq:y-Phy-xih-L2} \\
        & \norm{\mathscr{P}_h y + \xi_h}_{C([0,t_{R,\varrho}];\mathbb{L}^\infty)} \leqslant c R,  \label{eq:phy-xih-Linfty} \\
        & \norm{\nabla(\mathscr{P}_h y + \xi_h)}_{C([0,t_{R,\varrho}];\mathbb{L}^{q_0})} \leqslant c R,  \label{eq:phy-xih-H1q}
  \end{align}
  where \( q_0 \in (2, \frac{2}{2-\varrho}) \).
\end{lemma}
\begin{proof}
We split the proof into the following two steps.

\textbf{Step 1.}
Let \( u \in \dot{H}^{\varrho,2} \) and define \( u_h := A_h^{-1}\mathscr{P}_h A_2 u \). By definition,
\[
\langle \nabla(u - u_h), \nabla v_h \rangle = 0, \quad \forall v_h \in \mathbb{L}_{h,\mathrm{div}}.
\]
Using \cite[Theorem~3.1]{Brezzi1991} and arguments from \cite[Theorems~1.1 and 1.2, Chapter II]{Girault1986}, we obtain
$ \norm{u - u_h}_{\mathbb{L}^2} \leqslant ch^{\varrho} \norm{u}_{\dot{H}^{\varrho,2}} $.
Since this holds for all \( u \in \dot{H}^{\varrho,2} \), it follows that
\[
\norm{I - A_h^{-1}\mathscr{P}_h A_2}_{\mathcal{L}(\dot{H}^{\varrho,2}, \mathbb{L}^2)} \leqslant ch^{\varrho}.
\]
Combining this with \cref{lem:scrPh}(iii) yields
\begin{equation}
\norm{\mathscr{P}_h A_2 - A_h \mathscr{P}_h}_{\mathcal{L}(\dot{H}^{\varrho,2}, \dot{H}_h^{-2,2})} \leqslant c h^{\varrho}. \label{eq:I-Ah-Ph-A}
\end{equation}
Together with the inverse estimate \cref{eq:inverse}, this implies
\begin{equation}
\norm{\mathscr{P}_h A_2 - A_h \mathscr{P}_h}_{\mathcal{L}(\dot{H}^{\varrho,2}, \dot{H}_h^{\varrho-2,2})} \leqslant c. \label{eq:I-Ah-Ph-A2}
\end{equation}

\textbf{Step 2.}
For any \( g \in L^\infty(0,T;\dot H^{\varrho,2}) \), following the argument in \cite[Theorem~3.7]{Thomee2006},
we deduce by \cref{lem:Sh}(i), \cref{eq:I-Ah-Ph-A}, and \cref{eq:I-Ah-Ph-A2} that
\begin{align*}
& \norm{S_h \ast (\mathscr{P}_hA_2g - A_h\mathscr{P}_hg)}_{C([0,T];\mathbb{L}^2)} \\
  \leqslant{} & 
  \sup_{t\in [0,T]} \bigg[
     \int_0^{t\wedge h^2}
   \norm{S_h(t-s) \big(\mathscr P_hA_2 - A_h\mathscr{P}_h\big)g(s)}_{\mathbb L^2} 
   \, \mathrm{d}s \\
  & \qquad\qquad\qquad {} + \int_{t\wedge h^2}^t
   \norm{S_h(t-s) \big(\mathscr P_hA_2 - A_h\mathscr{P}_h\big)g(s)}_{\mathbb L^2} 
   \, \mathrm{d}s
   \bigg] \\
  \leqslant{} & 
  c\sup_{t\in [0,T]}\bigg[ \int_0^{t\wedge h^2}
   (t-s)^{\varrho/2-1} \, \mathrm{d}s +
  \int_{t\wedge h^2}^t
  (t-s)^{-1} h^\varrho \, \mathrm{d}s\bigg]
  \times \norm{g}_{L^\infty(0,T;\dot H^{\varrho,2})} \\
  \leqslant{} &
  ch^\varrho \Big(\ln\frac1h\Big) \norm{g}_{L^\infty(0,T;\dot H^{\varrho,2})}.
\end{align*}
The estimate \cref{eq:y-Phy-xih-L2} follows by employing the above result,
Lemma \ref{lem:scrPh}(iii), and the definition of the stopping time \(t_{R,\varrho}\) in \eqref{eq:tR}.
Let \( \mathcal I_h: \mathcal H^{1,2} \to \mathbb L_{h,\mathrm{div}} \) be the quasi-local interpolation operator from \cite[Theorem~3.1]{Girault2003}. Using \cite[Theorem~3.1]{Girault2003}
and the embeddings
\[ \dot H^{\varrho,2} \hookrightarrow \mathbb H^{1,2/(2-\varrho)} \quad \text{and} \quad \dot H^{\varrho,2} \hookrightarrow \mathbb L^\infty, \]
we conclude that the following estimates hold almost surely on the event \( \{\norm{y_0}_{\dot H^{\varrho,2}} \leqslant R\} \):
\begin{align*}
  &\norm{(I - \mathcal I_h)y}_{C([0,t_{R,\varrho}];\mathbb L^2)} \leqslant ch^\varrho R, \\
  &\norm{\mathcal I_hy}_{C([0,t_{R,\varrho}];\mathbb L^\infty)} \leqslant cR, \\
  &\norm{\nabla\mathcal I_hy}_{C([0,t_{R,\varrho}];\mathbb L^{2/(2-\varrho)})} \leqslant cR.
\end{align*}
We also have the following standard inverse estimates
(see, e.g., \cite[Theorem~4.5.11]{Brenner2008}):
for any $ u_h \in \mathbb L_{h,\mathrm{div}} $,
\begin{align*}
  \norm{u_h}_{\mathbb L^\infty} &\leqslant ch^{-1}\norm{u_h}_{\mathbb L^2}, \\
  \norm{\nabla u_h}_{\mathbb L^{q_0}} &\leqslant ch^{2/q_0-2}\norm{u_h}_{\mathbb L^2}, \quad 2 < q_0 < 2/(2-\varrho).
\end{align*}
Combining these estimates with \cref{eq:y-Phy-xih-L2},
the stability estimates \cref{eq:phy-xih-Linfty,eq:phy-xih-H1q} follow by
straightforward computation. This concludes the proof.
\end{proof}

Finally, we are in a position to prove \cref{thm:yh-regu,thm:yR-yhR,thm:y-yh-global} as follows.

\medskip\noindent\textbf{Proof of \cref{thm:yh-regu}.}
For brevity, in what follows, when we state that a quantity is bounded,
we mean that it is uniformly bounded with respect to \( h \).
We divide the proof into the following three steps.

\textbf{Step 1.}
Using the initial condition \( y_0 \in L_{\mathcal{F}_0}^4(\Omega; \mathbb{L}_{\mathrm{div}}^2) \),
the growth property of $ F $ in \cref{lem:F}(ii),
and the property that \( \langle G(u_h), u_h \rangle = 0 \) for all \( u_h \in \mathbb{L}_{h,\mathrm{div}} \), we apply Itô's formula to establish that  
\[
\norm{y_h}_{L^4(\Omega \times (0,T); \mathbb{L}^2)} + \norm{y_h}_{L^2(\Omega \times (0,T); \dot{H}_h^{1,2})} \quad \text{is bounded.}
\]  
By using Hölder's inequality and the inequality  
$ \norm{u_h}_{\dot{H}_h^{1/2,2}} \leqslant
\norm{u_h}_{\mathbb{L}^2}^{1/2} \norm{u_h}_{\dot{H}_h^{1,2}}^{1/2} $
for all $ u_h \in \mathbb{L}_{h,\mathrm{div}} $,
it follows that
\[
  y_h \text{ is bounded in \( L^{8/3}(\Omega \times (0,T); \dot{H}_h^{1/2,2}) \)}.  
\]
Invoking Lemma \ref{lem:F}(iii), Lemma \ref{lem:Ph}(vi),
and the ideal property of the \(\gamma\)-radonifying
operators (see \cite[Theorem~9.1.10]{HytonenWeis2017}),
we deduce that \( \mathcal{P}_h F(y_h) \)
is bounded in \( L^{8/3}(\Omega \times (0,T); \gamma(H,\dot{H}_h^{\alpha,2})) \) for all \( 0 < \alpha < 1/2 \).  
By Lemma \ref{lem:Sh}(iii) and the \( h \)-uniform embedding  
\[
  (\dot{H}_h^{\alpha,2}, \dot{H}_h^{2+\alpha,2})_{1/8,8/3} \hookrightarrow \dot{H}_h^{0.7,2} 
  \quad\text{(valid for $ \alpha \text{ near } 1/2$)},
\]
justified via Propositions 1.3–1.4 and Theorem 4.36 in \cite{Lunardi2018}, we obtain that
\[
  S_h \diamond \mathcal{P}_h F(y_h) \text{ is bounded in }
  L^{8/3}(\Omega; C([0,T]; \dot{H}_h^{0.7,2})).
\]  
Furthermore, Lemma \ref{lem:Sh}(i), Lemma \ref{lem:Ph}(ii),
and the initial condition \( y_0 \in L_{\mathcal{F}_0}^1(\Omega; \dot{H}^{\varrho,2}) \) imply that
\[
z_h := S_h(\cdot) \mathcal{P}_h y_0 + S_h \diamond \mathcal{P}_h F(y_h)
\text{ is bounded in } L^1(\Omega; C([0,T]; \dot{H}_h^{0.7,2})).
\]  
Let \( \eta_h := y_h - z_h \). We have, almost surely,  
$ \frac{\mathrm{d}}{\mathrm{d}t} \eta_h(t) = -A_h \eta_h(t) - \mathcal{P}_h G(y_h(t)) $
for all $ t \in [0,T] $. Using the inequality  
\[
\norm{\mathcal{P}_h G(u_h)}_{\dot{H}_h^{-0.3,2}} \leqslant c \norm{u_h}_{\dot{H}_h^{0.7,2}} \norm{u_h}_{\dot{H}_h^{1,2}} \quad \text{for all } u_h \in \mathbb{L}_{h,\mathrm{div}},
\]  
which follows directly from Hölder's inequality,
the embeddings $\mathbb H^{0.7,2} \hookrightarrow \mathbb L^{20/3} $
and $ \mathbb H^{0.3,2} \hookrightarrow \mathbb L^{20/7} $,
and \cref{eq:1}, we follow the argument in Step 1 of the proof of \cref{prop:regu}
to deduce that
\begin{equation}
\label{eq:yh-regu-0}
\mathbb{P}\big( \norm{y_h}_{C([0,T]; \dot{H}_h^{0.7,2})} \leqslant R_h \big) \geqslant 1 - \frac{c}{\ln(1 + R_h)}, \quad \forall R_h > 1.
\end{equation}

\textbf{Step 2.}
We now establish that for any \(\alpha \in (1/2,1)\),
\begin{align}
  \norm{\mathcal P_hG(u_h)}_{\dot H_h^{2\alpha-2,2}}
  \leqslant c \norm{u_h}_{\dot H_h^{\alpha,2}}^2,
  \quad \forall u_h \in \mathbb L_{h,\mathrm{div}}.
  \label{eq:PhG}
\end{align}
To proceed, recall the standard discrete Sobolev inequality (cf.~\cref{rem:discrete_sobolev} for related details):
\begin{align}
  \norm{\nabla w_h}_{\mathbb L^{2/(2-\beta)}}
  \leqslant c \norm{w_h}_{\dot H_h^{\beta,2}},
  \quad \forall w_h \in \mathbb L_{h,\mathrm{div}},
  \, \forall \beta \in [1,2).
  \label{eq:discrete_sobolev}
\end{align}
Fix \( 1/2 < \alpha < 1 \). Using Hölder's inequality,
the embeddings \( \mathbb H^{\alpha,2}\hookrightarrow\mathbb L^{2/(1-\alpha)} \) and \( \mathbb H^{1-\alpha,2}\hookrightarrow\mathbb L^{2/\alpha} \), along with \cref{lem:P}, \cref{eq:1}, and \cref{eq:discrete_sobolev},
we derive for any \( u_h, v_h, w_h \in \mathbb L_{h,\mathrm{div}} \):
\begin{align*}
  & \dual{(\mathcal Pu_h\cdot\nabla) v_h, \, w_h} =
  \dual{(\mathcal Pu_h\cdot\nabla) w_h, \, v_h} \leqslant
  c\norm{u_h}_{\dot H_h^{\alpha,2}} \norm{v_h}_{\dot H_h^{1-\alpha,2}} \norm{w_h}_{\dot H_h^{1,2}}, \\
  & \dual{(\mathcal Pu_h\cdot\nabla) v_h, \, w_h} \leqslant c\norm{u_h}_{\dot H_h^{\alpha,2}} \norm{v_h}_{\dot H_h^{2-\alpha,2}} \norm{w_h}_{\dot H_h^{0,2}}.
\end{align*}
Consequently, interpolation (\cite[Theorem~2.6]{Lunardi2018}) yields
\begin{align}
  \dual{(\mathcal Pu_h\cdot\nabla) u_h, \, v_h}
  \leqslant c\norm{u_h}_{\dot H_h^{\alpha,2}}^2
  \norm{v_h}_{\dot H_h^{2-2\alpha,2}},
  \quad \forall u_h, v_h \in \mathbb L_{h,\mathrm{div}}.
  \label{eq:lxy-1}
\end{align}
For any \( u_h, v_h \in \mathbb{L}_{h,\mathrm{div}} \), the following estimate holds:
\begin{align}  
   \dual{\big[(u_h - \mathcal{P}u_h) \cdot \nabla\big] u_h, \, v_h} 
  &\stackrel{\mathrm{(i)}}{\leqslant} 
  c \norm{u_h - \mathcal{P}u_h}_{\mathbb{L}^2}  
  \norm{u_h}_{\dot{H}_h^{2\alpha,2}}  
  \norm{v_h}_{\mathbb{H}^{2-2\alpha,2}} \notag \\  
  &\stackrel{\mathrm{(ii)}}{\leqslant} 
  c \norm{u_h}_{\dot{H}_h^{\alpha,2}}^2 \norm{v_h}_{\dot{H}_h^{2-2\alpha,2}},  
  \label{eq:lxy-2}  
\end{align}  
where (i) follows from Hölder's inequality,
the embedding $ \mathbb H^{2-2\alpha,2} \hookrightarrow \mathbb L^{2/(2\alpha-1)} $,
and assertion \cref{eq:discrete_sobolev}, and (ii) is a consequence of \cref{lem:Ph}(v), \cref{eq:inverse}, and \cref{eq:1}.
Similarly, for any \( u_h, v_h \in \mathbb{L}_{h,\mathrm{div}} \), we have  
\begin{align}  
  \dual{(\nabla \cdot u_h) u_h, \, v_h}
   & = 
   \dual{\nabla \cdot u_h, \, u_h \cdot v_h} \notag \\  
   &\stackrel{\mathrm{(i)}}{=}{} 
   \dual{\nabla \cdot u_h, \, u_h \cdot v_h - \Pi_h(u_h \cdot v_h)} \notag \\  
   &\stackrel{\mathrm{(ii)}}{\leqslant}
   c h \norm{\nabla u_h}_{\mathbb{L}^2} \norm{\nabla(u_h \cdot v_h)}_{\mathbb{L}^2} \notag \\  
   &\stackrel{\mathrm{(iii)}}{\leqslant}
   c h \norm{\nabla u_h}_{\mathbb{L}^2} 
   \big(
     \norm{u_h}_{\dot{H}_h^{2\alpha,2}} \norm{v_h}_{\dot{H}_h^{2-2\alpha,2}}
     + \norm{u_h}_{\dot{H}_h^{\alpha,2}} \norm{v_h}_{\dot{H}_h^{2-\alpha,2}}
   \big) \notag \\  
   &\stackrel{\mathrm{(iv)}}{\leqslant}
   c \norm{u_h}_{\dot{H}_h^{\alpha,2}}^2  
   \norm{v_h}_{\dot{H}_h^{2-2\alpha,2}},  
   \label{eq:lxy-3}  
\end{align}
where in (i), \( \Pi_h(u_h \cdot v_h) \) denotes the \( L^2 \)-orthogonal projection of \( u_h \cdot v_h \) onto \( M_h \); (ii) relies on the standard approximation estimate  
$ \norm{u_h \cdot v_h - \Pi_h(u_h \cdot v_h)}_{L^2} \leqslant c h \norm{\nabla(u_h \cdot v_h)}_{\mathbb{L}^2}  $;
(iii) follows from Hölder's inequality, Sobolev's embedding theorem, \cref{eq:1}, and \cref{eq:discrete_sobolev}; and (iv) is derived from the inverse estimate \cref{eq:inverse}.
For any \( u_h \in \mathbb L_{h,\mathrm{div}} \), since \cref{eq:G-def} implies  
\[
G(u_h) = (\mathcal Pu_h\cdot\nabla)u_h  
+ \big[(u_h-\mathcal Pu_h)\cdot\nabla\big]u_h  
+ \frac12 (\nabla\cdot u_h)u_h,  
\]  
combining \cref{eq:lxy-1,eq:lxy-2,eq:lxy-3} yields  
\[
\dual{\mathcal P_hG(u_h), \, v_h} \leqslant c\norm{u_h}_{\dot H_h^{\alpha,2}}^2  
\norm{v_h}_{\dot H_h^{2-2\alpha,2}},  
\quad \forall v_h \in \mathbb L_{h,\mathrm{div}}.  
\]  
This establishes \cref{eq:PhG}.

\textbf{Step 3.}
By the stability properties of \(\mathcal{P}_h\) (Lemma \ref{lem:Ph}(ii,vi)) and the properties of \(S_h\) (Lemma \ref{lem:Sh}), the inequality \eqref{eq:yh-regu} follows directly from \eqref{eq:yh-regu-0} and \eqref{eq:PhG} via the methodology in Step 2 of Proposition \ref{prop:regu}. This completes the proof of Theorem \ref{thm:yh-regu}.

\hfill$\blacksquare$

\medskip\noindent\textbf{Proof of \cref{thm:yR-yhR}.}
We split the proof into the following several steps.

\textbf{Step 1.}
Let \( e_h := y_h - \mathscr P_hy \) and introduce the decomposition \( e_{h} = \xi_{h} + \eta_h \), where
\[
  \xi_h := S_h \ast (\mathscr P_hA_2y - A_h\mathscr P_hy).
\]
Given that $ \xi_{h}'(t) = -A_h\xi_{h}(t) + \mathscr{P}_hA_2y(t) - A_h\mathscr{P}_hy(t) $
almost surely for all $ t \in [0,T] $,
a direct computation yields
\begin{align*}
  \mathrm{d}\eta_{h,R}(t)
    &= \mathbbm{1}_{[0,t_{R,\varrho}]}(t) \Big\{ -A_h\eta_{h}(t) \, \mathrm{d}t
      + \big[\mathscr P_h\mathcal{P}G(y(t)) - \mathcal P_hG(y_h(t))\big]  \, \mathrm{d}t \\
    & \qquad\qquad\qquad\quad {}
  + \big[ \mathcal P_hF(y_h(t)) - \mathscr P_h\mathcal PF(y(t)) \big]\,\mathrm{d}W_H(t) \Big\},
  \quad t \in [0,T],
\end{align*}
where $ \eta_{h,R}(t) := \eta_h(t \wedge t_{R,\varrho}) $ for all $ t \in [0,T] $.
Applying Itô's formula to \( \|\eta_{h,R}(t)\|_{\mathbb{L}^2}^p\),
we infer the following equality for all \( t \in [0,T] \):
\begin{align*}
  \mathbb{E}\|\eta_{h,R}(t)\|_{\mathbb{L}^2}^p
    &= \mathbb E\norm{(\mathcal P_h-\mathscr P_h)y_0}_{\mathbb L^2}^p
    + p \, \mathbb{E} \left[
      \int_0^t \|\eta_{h,R}(s)\|_{\mathbb{L}^2}^{p-2}
      \big( I_1(s) + \cdots + I_5(s) \big) \, \mathrm{d}s
    \right],
\end{align*}
where
\begin{align*}
  I_1(s) &:= \mathbbm{1}_{[0,t_{R,\varrho}]}(s) \big\langle \eta_{h,R}(s), \, -A_h\eta_{h}(s) \big\rangle, \\
  I_2(s) &:= \mathbbm{1}_{[0,t_{R,\varrho}]}(s) \big\langle \eta_{h,R}(s), \, \big(\mathscr P_h\mathcal{P} - I \big) G(y(s)) \big\rangle, \\
  I_3(s) &:= \mathbbm{1}_{[0,t_{R,\varrho}]}(s) \big\langle \eta_{h,R}(s), \, G(y(s)) - G(y_{h}(s)) \big\rangle, \\
  I_4(s) &:= \frac{1}{2} \mathbbm{1}_{[0,t_{R,\varrho}]}(s)
  \left\lVert \mathcal P_hF(y_{h}(s)) - \mathscr P_h\mathcal PF(y(s)) \right\rVert_{\gamma(H,\mathbb{L}^2)}^2, \\
  I_5(s) &:=  \frac{p-2}{2}\mathbbm{1}_{[0,t_{R,\varrho}]}(s) \|\eta_{h,R}(s)\|_{\mathbb{L}^2}^{-2}
  \left\lVert \bigl\langle \eta_{h,R}(s), \, \mathcal P_hF(y_h(s)) - \mathscr P_h\mathcal PF(y(s)) \bigr\rangle \right\rVert_{\gamma(H,\mathbb{R})}^2.
\end{align*}

\textbf{Step 2.}
We now estimate the terms \(I_1\) and \(I_2\).  
For \(I_1\), the definition of \(A_h\) implies  
\begin{equation}
  \label{eq:I1-bound}
  I_1(s) = -\mathbbm{1}_{[0,t_{R,\varrho}]}(s) \|\nabla \eta_{h,R}(s)\|^2_{\mathbb{L}^2}.
\end{equation}
To bound \(I_2\), we use \cref{lem:scrPh}(i), which gives  
\[
  I_2(s) \leqslant ch^\varrho \mathbbm{1}_{[0,t_{R,\varrho}]}(s) \|\nabla \eta_{h,R}(s)\|_{\mathbb{L}^2} \|G(y(s))\|_{\mathbb H^{\varrho-1,2}}.
\]
Given the incompressibility condition \( \nabla \cdot y(s) = 0 \), we have
\begin{align*}
  \|G(y(s))\|_{\mathbb H^{\varrho-1,2}} = \left\lVert\big[y(s) \cdot \nabla\big] y(s)\right\rVert_{\mathbb H^{\varrho-1,2}}.
\end{align*}
Considering the embeddings \( \dot{H}^{\varrho,2} \hookrightarrow \mathbb{L}^\infty \),
\( \dot H^{\varrho,2} \hookrightarrow W^{1,4/(3-\varrho)}(\mathcal O;\mathbb R^2) \),
and \( \dot H^{2,2} \hookrightarrow W^{1,4/(1+\varrho)}(\mathcal O;\mathbb R^2) \),
we can use the standard interpolation theory (cf.~Theorem 2.6 in \cite{Lunardi2018})
to deduce that
\begin{align*}
  \left\lVert\big[y(s)\cdot\nabla\big]y(s)\right\rVert_{\mathbb H^{\varrho-1,2}} \leqslant c \|y(s)\|_{\dot{H}^{\varrho,2}}^2.
\end{align*}
By the definition of $ t_{R,\varrho} $ in \cref{eq:tR}, it follows that
\begin{align*}
  \mathbbm{1}_{[0,t_{R,\varrho}]}(s)
  \left\lVert\big[y(s)\cdot\nabla\big]y(s)\right\rVert_{\mathbb H^{\varrho-1,2}}
  \leqslant c R^2,
  \quad\text{$\mathrm{d}\mathbb P\otimes\mathrm{d}s$-a.e.}
\end{align*}
Thus, we obtain the estimate
\[
  I_2(s) \leqslant ch^\varrho R^2\mathbbm{1}_{[0,t_{R,\varrho}]}(s) \|\nabla\eta_{h,R}(s)\|_{\mathbb{L}^2},
  \quad\text{$\mathrm{d}\mathbb P\otimes\mathrm{d}s$-a.e.}
\]
Applying Young's inequality, we further obtain
\begin{equation}
  I_2(s) \leqslant \mathbbm{1}_{[0,t_{R,\varrho}]}(s) \left(
    \frac{1}{4} \|\nabla \eta_{h,R}(s)\|_{\mathbb{L}^2}^2
    + c h^{2\varrho} R^4
  \right),
  \quad \text{$\mathrm{d}\mathbb P\otimes\mathrm{d}s$-a.e.}
  \label{eq:I2-bound}
\end{equation}

\textbf{Step 3.}
We now estimate the term \(I_3\). Using the definition of \(G\) in \cref{eq:G-def}, the incompressibility condition \(\nabla \cdot y(s) = 0\), and the decomposition
\(y_h = \mathscr{P}_h y + \xi_h + \eta_h\), integration by parts yields  
\[
\begin{aligned}
I_3(s) &= -\mathbbm{1}_{[0,t_{R,\varrho}]}(s) \left\langle (y-y_h)(s), \, \big(y(s)\cdot\nabla\big)\eta_{h,R}(s) \right\rangle \\
&\quad + \frac{1}{2}\mathbbm{1}_{[0,t_{R,\varrho}]}(s) \left\langle \eta_{h,R}(s), \, \big((y - y_h)(s) \cdot \nabla\big)(\mathscr{P}_h y + \xi_h)(s) \right\rangle \\
&\quad + \frac{1}{2} \mathbbm{1}_{[0,t_{R,\varrho}]}(s) \left\langle \big((y_h-y)(s)\cdot\nabla\big)\eta_{h,R}(s), \, (\mathscr{P}_h y + \xi_h)(s) \right\rangle \\
&=: J_1(s) + J_2(s) + J_3(s).
\end{aligned}
\]  
For \(J_1\), Hölder's inequality and the embedding \(\dot H^{\varrho,2} \hookrightarrow \mathbb L^\infty\) imply  
\[
J_1(s) \leqslant cR\mathbbm{1}_{[0,t_{R,\varrho}]}(s) \|\nabla \eta_{h,R}(s)\|_{\mathbb{L}^2} \|(y-y_h)(s)\|_{\mathbb L^2}, \quad \text{$\mathrm{d}\mathbb P\otimes\mathrm{d}s$-a.e.}
\]  
For \(J_2\) and \(J_3\), Hölder's inequality, the stability estimates for \(\mathscr{P}_h y + \xi_h\) in \cref{eq:phy-xih-Linfty,eq:phy-xih-H1q}, and the embedding \(\dot H^{1,2} \hookrightarrow \mathbb L^q\) for \(q \in (2,\infty)\) yield  
\[
J_2(s) + J_3(s) \leqslant cR\mathbbm{1}_{[0,t_{R,\varrho}]}(s) \|\nabla\eta_{h,R}(s)\|_{\mathbb L^2} \|(y-y_h)(s)\|_{\mathbb L^2}, \quad \text{$\mathrm{d}\mathbb P\otimes\mathrm{d}s$-a.e.}
\]  
Combining these bounds, we conclude  
\begin{equation}
\label{eq:I3-bound}
I_3(s) \leqslant \mathbbm{1}_{[0,t_{R,\varrho}]}(s) \Bigl( \frac{1}{4} \|\nabla \eta_{h,R}(s)\|_{\mathbb{L}^2}^2 + c R^2 \|(y - y_h)(s)\|_{\mathbb{L}^2}^2 \Bigr), \quad \text{$\mathrm{d}\mathbb P\otimes\mathrm{d}s$-a.e.}
\end{equation}

\textbf{Step 4.}
We now estimate the sum \(I_4 + I_5\). Observe that  
\[
\begin{aligned}
I_4(s) + I_5(s) &\leqslant c \, \mathbbm{1}_{[0,t_{R,\varrho}]}(s) \left\lVert \mathcal{P}_h(F(y_h(s)) - F(y(s))) \right\rVert_{\gamma(H,\mathbb{L}^2)}^2 \\
&\quad + c \, \mathbbm{1}_{[0,t_{R,\varrho}]}(s) \left\lVert \mathcal{P}_h(F(y(s)) - \mathcal{P}F(y(s))) \right\rVert_{\gamma(H,\mathbb{L}^2)}^2 \\
&\quad + c \, \mathbbm{1}_{[0,t_{R,\varrho}]}(s) \left\lVert (\mathcal{P}_h - \mathscr{P}_h)\mathcal{P}F(y(s)) \right\rVert_{\gamma(H,\mathbb{L}^2)}^2 \\
&=: K^{(1)}(s) + K^{(2)}(s) + K^{(3)}(s).
\end{aligned}
\]  
For \(K^{(1)}(s)\),
using the Lipschitz continuity of $F$ in $\gamma(H,\mathbb{L}^2)$
(cf.~Lemma \ref{lem:F}(i)) and the ideal property (Theorem 9.1.10 in \cite{HytonenWeis2017}),
we get
\[
K^{(1)}(s) \leqslant c \, \mathbbm{1}_{[0,t_{R,\varrho}]}(s) \|(y - y_h)(s)\|_{\mathbb{L}^2}^2.
\]  
For \(K^{(2)}(s)\), using \cref{eq:919}, \cref{lem:F}(iv),
and the ideal property, we derive  
\[
\begin{aligned}
K^{(2)}(s) &\leqslant c h^{2\varrho} \, \mathbbm{1}_{[0,t_{R,\varrho}]}(s) \|F(y(s))\|_{\gamma(H,\mathbb{H}^{\varrho,2})}^2 \\
&\leqslant c h^{2\varrho} R^{2\varrho} \, \mathbbm{1}_{[0,t_{R,\varrho}]}(s), \quad \text{$\mathrm{d}\mathbb{P} \otimes \mathrm{d}s$-a.e.}
\end{aligned}
\]  
Similarly, for \(K^{(3)}(s)\), by \cref{lem:scrPh}(ii), \cref{lem:P},
and the ideal property, we have  
\[
\begin{aligned}
K^{(3)}(s) &\leqslant c h^{2\varrho} \, \mathbbm{1}_{[0,t_{R,\varrho}]}(s) \|F(y(s))\|_{\gamma(H,\mathbb{H}^{\varrho,2})}^2 \\
&\leqslant c h^{2\varrho} R^{2\varrho} \, \mathbbm{1}_{[0,t_{R,\varrho}]}(s), \quad \text{$\mathrm{d}\mathbb{P} \otimes \mathrm{d}s$-a.e.}
\end{aligned}
\]  
Combining these results, we conclude  
\begin{equation}
  \label{eq:I4-bound}
  I_4(s) + I_5(s) \leqslant c \, \mathbbm{1}_{[0,t_{R,\varrho}]}(s) \big( \|(y - y_h)(s)\|_{\mathbb{L}^2}^2 + h^{2\varrho} R^{2\varrho} \big), \quad \text{$\mathrm{d}\mathbb{P} \otimes \mathrm{d}s$-a.e.}
\end{equation}


\textbf{Step 5.} 
Integrating the estimates for \(I_1, \ldots, I_5\) from
\cref{eq:I1-bound,eq:I2-bound,eq:I3-bound,eq:I4-bound},
along with the inequality 
\begin{equation}
  \label{eq:y0-err}
  \mathbb E\norm{(\mathcal P_h-\mathscr P_h)y_0}_{\mathbb L^2}^p
  \leqslant ch^{\varrho p},
\end{equation}
which is justified by \cref{lem:scrPh}(ii) and the initial condition $ y_0 \in L_{\mathcal F_0}^p(\Omega;\dot H^{\varrho,2}) $,
and applying Young's inequality, we conclude that for every \(t \in [0, T]\),  
\[
\begin{aligned}
& \mathbb{E} \|\eta_{h,R}(t)\|_{\mathbb{L}^2}^p + \mathbb{E} \bigg[ \int_0^t \mathbbm{1}_{[0,t_{R,\varrho}]}(s) \|\eta_{h,R}(s)\|_{\mathbb{L}^2}^{p-2} \|\nabla\eta_{h,R}(s)\|_{\mathbb{L}^2}^2 \, \mathrm{d}s \bigg] \\
\leqslant{}
&cR^{p+2}h^{\varrho p} + cR^2\mathbb{E} \bigg[ \int_0^t \mathbbm{1}_{[0,t_{R,\varrho}]}(s) \Big( \|\eta_{h,R}(s)\|_{\mathbb{L}^2}^p + \|(y-y_h)(s)\|_{\mathbb{L}^2}^p \Big) \, \mathrm{d}s \bigg].
\end{aligned}
\]  
Using the decomposition \(y - y_h = y - \mathscr{P}_h y - \xi_h - \eta_h\) and the inequality \cref{eq:y-Phy-xih-L2}, we infer that for any \(t \in [0,T]\),  
\[
\begin{aligned}
& \mathbb{E} \|\eta_{h,R}(t)\|_{\mathbb{L}^2}^p + \mathbb{E} \bigg[ \int_0^t \mathbbm{1}_{[0,t_{R,\varrho}]}(s) \|\eta_{h,R}(s)\|_{\mathbb{L}^2}^{p-2} \|\nabla\eta_{h,R}(s)\|_{\mathbb{L}^2}^2 \, \mathrm{d}s \bigg] \\
&\leqslant cR^{p+2}h^{\varrho p} \big(\ln\frac{1}{h}\big)^p + cR^{2} \, \mathbb{E} \left[ \int_0^t \mathbbm{1}_{[0,t_{R,\varrho}]}(s) \|\eta_{h,R}(s)\|_{\mathbb{L}^2}^p \, \mathrm{d}s \right].
\end{aligned}
\]  
Applying Gronwall's inequality, we obtain the bound  
\begin{equation}
\label{eq:etah-C}
\|\eta_{h,R}\|_{C([0,T];L^p(\Omega;\mathbb{L}^2))} + \left\lVert \|\eta_{h}\|_{\mathbb{L}^2}^{1-\frac{2}{p}} \|\nabla\eta_{h}\|_{\mathbb{L}^2}^{\frac{2}{p}} \right\rVert_{L^p(\Omega\times(0,t_{R,\varrho}))} \leqslant c h^{\varrho} \ln\frac{1}{h} \exp(c R^2).
\end{equation}


  \textbf{Step 6.}
  Applying Itô's formula to
  \( \frac{1}{2}\|\eta_h(\cdot \wedge t_{R,\varrho})\|_{\mathbb L^2}^2 \),
  we obtain almost surely that for any \( t \in [0, T] \):
  \begin{align*}
    \frac{1}{2}\|\eta_h(t \wedge t_{R,\varrho})\|_{\mathbb{L}^2}^2
  &= \frac12\norm{(\mathcal P_h-\mathscr P_h)y_0}_{\mathbb L^2}^2
  + \overbrace{\int_0^t (I_1(s) + I_2(s) + I_3(s) + I_4(s)) \, \mathrm{d}s}^{II_1(t)} \\
  &\quad {} + \overbrace{\int_0^t \mathbbm{1}_{[0,t_{R,\varrho}]}(s) \left\langle \eta_h(s), \, \mathcal P_hF(y_h(s)) - \mathscr{P}_h\mathcal PF(y(s)) \right\rangle \, \mathrm{d}W_H(s)}^{II_2(t)}.
  \end{align*}
  For \( II_1 \), using the bounds from \cref{eq:I1-bound,eq:I2-bound,eq:I3-bound,eq:I4-bound},
  we deduce
  \begin{align*}
    \mathbb{E} \|II_1\|_{C([0, T])}^{p/2}
    \leqslant cR^{2p} h^{\varrho p} + c R^p \left\lVert y-y_h\right\rVert_{L^p(\Omega\times(0,t_{R,\varrho});\mathbb{L}^2)}^p.
  \end{align*}
  For \( II_2 \), an argument analogous to the one in \cref{eq:I4-bound} yields
  \begin{align*}
    & \mathbbm{1}_{[0,t_{R,\varrho}]}(s)
    \left\lVert \left\langle \eta_h(s), \, \mathcal P_hF(y_h(s)) -\mathscr P_h \mathcal{P}F(y(s)) \right\rangle \right\rVert_{\gamma(H,\mathbb{R})}^2 \\
    \leqslant{}
    & c \|\eta_h(s)\|_{\mathbb{L}^2}^2 \big( \| (y - y_h)(s) \|_{\mathbb{L}^2}^2 + h^{2\varrho}R^{2\varrho} \big),
    \quad \text{$\mathrm{d}\mathbb P\otimes\mathrm{d}s$-a.e.}
  \end{align*}
  Hence, using the Burkholder-Davis-Gundy inequality and Hölder's inequality, we obtain
  \begin{align*}
    \mathbb{E} \|II_2\|_{C([0, T])}^{p/2}
    \leqslant cR^{\varrho p}h^{\varrho p} + c\norm{\eta_h}_{L^p(\Omega\times(0,t_{R,\varrho});\mathbb L^2)}^p
    + c \norm{y-y_h}_{L^p(\Omega\times(0,t_{R,\varrho});\mathbb L^2)}^p.
  \end{align*}
  Combining the estimates for \( II_1 \) and \( II_2 \),
  along with \cref{eq:y0-err}, we obtain
  \begin{align*}
    \norm{\eta_h(\cdot\wedge t_{R,\varrho})}_{L^p(\Omega;C([0,T];\mathbb L^2))}
    \leqslant cR^{2}h^{\varrho} + c\norm{\eta_h}_{L^p(\Omega\times(0,t_{R,\varrho});\mathbb L^2)}
    + cR \norm{y-y_h}_{L^p(\Omega\times(0,t_{R,\varrho});\mathbb L^2)}.
  \end{align*}
Combining this inequality with \cref{eq:etah-C,eq:y-Phy-xih-L2},
and using the identity \( y - y_h = y - \mathscr{P}_h y - \xi_h - \eta_h \),
we obtain the desired error estimate \cref{eq:yR-yhR}.
This completes the proof of \cref{thm:yR-yhR}.

\hfill$\blacksquare$

\medskip\noindent
\textbf{Proof of \cref{thm:y-yh-global}.}
Given that \( \langle G(u_h), u_h \rangle = 0 \) for all \( u_h \in \mathbb{L}_{h,\mathrm{div}} \), and considering \( \|\mathcal{P}_hF(u_h)\|_{\gamma(H,\dot{H}_h^{0,2})} \leqslant c(1 + \|u_h\|_{\mathbb{L}^2}) \) for all \( u_h \in \mathbb{L}_{h,\mathrm{div}} \)
(as justified by Lemma \ref{lem:F}(ii) and the ideal property \cite[Theorem~9.1.10]{HytonenWeis2017}), an application of Itô's formula, combined with the Burkholder-Davis-Gundy inequality, leads to the uniform boundedness of \( \|y_h\|_{L^4(\Omega; C([0,T]; \mathbb{L}^2))} \) with respect to \( h \). This, together with (\ref{eq:y-L4-C-L2}), implies the uniform boundedness of \( \|y - y_h\|_{L^4(\Omega; C([0,T]; \mathbb{L}^2))} \). Furthermore, we derive the estimate:
\begin{align*}
    \|y - y_h\|_{L^2(\Omega; C([0,T]; \mathbb{L}^2))} 
    &\leqslant \|y - y_h\|_{L^2(\{t_{R,\varrho} = T\}; C([0,T]; \mathbb{L}^2))} \\
    &\quad + \left( \mathbb{P}(t_{R,\varrho}<T) \right)^{1/4} \|y - y_h\|_{L^4(\Omega; C([0,T]; \mathbb{L}^2))},
\end{align*}
where $ t_{R,\varrho} $ is defined by \cref{eq:tR} with $ R > 1 $.
By (\ref{eq:P-tRrho}) and (\ref{eq:yR-yhR}), there exist constants \( c_0, c_1, c_2 > 0 \), independent of \( h \) and \( R \), such that:
\begin{align*}
  \|y - y_h\|_{L^2(\Omega; C([0,T]; \mathbb{L}^2))}
  \leqslant c_0 h^{\varrho}\ln\frac{1}{h} \exp(c_1 R^2) + \frac{c_2}{\left(\ln(1+R)\right)^{1/4}}.
\end{align*}
Selecting 
\[
  R = \sqrt{ \frac{\varrho\ln\frac{1}{h} - \ln\left(\ln\frac{1}{h}\right) - \frac{1}{4}\ln\left(\ln\left(\ln\frac{1}{h}\right)\right)}{c_1} },
\]
and choosing \( h_0 \in (0,\exp(-3)) \) sufficiently small, we obtain the error estimate (\ref{eq:y-yh-global}) for all \( 0 < h < h_0 \). This completes the proof of Theorem \ref{thm:y-yh-global}.

\hfill$\blacksquare$

  \section{Full Discretization}
  \label{sec:full-discretization}
  Let $J$ be a positive integer and denote the time step by $\tau := T/J$.
  This section is devoted to analyzing the following full discretization:
  \begin{equation}
    \label{eq:Y}
    \begin{cases}
      \displaystyle
      Y_{j+1} - Y_j = -\tau \big(A_h Y_{j+1} + \mathcal P_h G(Y_{j+1})\big)
       + \mathcal P_h \int_{{j\tau}}^{{j\tau+\tau}} F(Y_j) \, \mathrm{d}W_H(t), \, 0 \leqslant j < J, \\
      Y_0 = \mathcal P_h y_0,
    \end{cases}
  \end{equation}
  where $ G $ is defined by \cref{eq:G-def}.
  For a thorough discussion of the solvability and stability properties of this discretization, we refer the reader to Section 3 of \cite{Prohl2013}.

  We begin our analysis by establishing the pathwise uniform convergence of the full discretization
  \cref{eq:Y} to the spatial semidiscretization \cref{eq:yh} on a specified local probability set.
  The following theorem formalizes this result.

  \begin{theorem}
    \label{thm:yhR-YR}
    Let \( y_0 \in L_{\mathcal{F}_0}^2(\Omega;\mathbb{L}_{\mathrm{div}}^2) \cap L_{\mathcal{F}_0}^1(\Omega;\dot{H}^{\varrho,2}) \) for some \( \varrho \in (1, 3/2) \), and let \( h, \tau \in (0, 1/2) \). Suppose \( y_h \) is the strong solution of the spatial semidiscretization \cref{eq:yh}, and \( (Y_j)_{j=0}^J \) is the solution to the full discretization \cref{eq:Y}. For any \( R_h, R_{h,\tau} > 1 \), define the local probability set  
    \begin{equation}
      \label{eq:local_sample_set}
      \Omega_{\varrho,R_h,R_{h,\tau}}^{h,\tau} := \left\{
        \omega \in \Omega : \, \norm{y_h}_{C([0,T];\dot{H}_h^{\varrho,2})} \leqslant R_h, \,
        \max_{0 \leqslant j \leqslant J} \|Y_j\|_{\mathbb{L}^2} \leqslant R_{h,\tau}
      \right\}.
    \end{equation}
    Then, there exists a constant \( c > 0 \), independent of \( h \), \( \tau \), \( R_h \), and \( R_{h,\tau} \), such that  
    \begin{equation}
      \label{eq:local-conv}
      \mathbb{E} \left[
        \mathbbm{1}_{\Omega_{\varrho,R_h,R_{h,\tau}}^{h,\tau}}
        \max_{1 \leqslant j \leqslant J} \norm{y_h(j\tau) - Y_j}_{\mathbb{L}^2}^2
      \right] \leqslant c\tau \exp(cR_h^2) \left( 1 + \ln\left(\frac{1}{h}\right) R_{h,\tau}^2 \right).
    \end{equation}
  \end{theorem}
\begin{proof}
We introduce the stopping time
\begin{equation}
  \label{eq:tR-new}
  t_{R_h}^{h,\varrho} := \inf\{t \in [0,T]: \, \norm{y_h(t)}_{\dot H_h^{\varrho,2}} \geqslant R_h\},
\end{equation}
with the convention that \( t_{R_h}^{h,\varrho} = T \) if the set is empty.
Motivated by equation (4.5) from \cite{Breit2023}, we define the sequence
\( (Y_{j,R_h})_{j=1}^J \) by
\begin{equation*}
  \begin{cases}
    Y_{j+1,R_h} - Y_{j,R_h} = -  \displaystyle\int_{{j\tau}}^{{j\tau+\tau}} \mathbbm{1}_{[0,t_{R_h}^{h,\varrho}]}(t)
    \bigl[ A_h Y_{j+1,R_h} + \mathcal{P}_h G(Y_{j+1,R_h}) \bigr] \, \mathrm{d}t \\
    \qquad\qquad\qquad\qquad {} + \mathcal{P}_h \displaystyle\int_{{j\tau}}^{{j\tau+\tau}}
     \mathbbm{1}_{[0,t_{R_h}^{h,\varrho}]}(t) \, F(Y_{j,R_h}) \, \mathrm{d}W_H(t), \quad\quad \text{for } 0 \leqslant j < J, \\
    Y_{0,R_h} = \mathcal{P}_h y_0.
  \end{cases}
\end{equation*}
By \cite[Theorem~7.2]{Baldi2017}, \( Y_j = Y_{j,R_h} \) holds almost surely on the
event \( \{{j\tau} \leqslant t_{R_h}^{h,\varrho}\} \) for all $ 1 \leqslant j \leqslant J $.
We also define the truncated semidiscrete solution
\begin{equation}
  \label{eq:yhR1}
  y_{h,R_h}(t) := y_h(t \wedge t_{R_h}^{h,\varrho}), \quad \text{for } 0 \leqslant t \leqslant T.
\end{equation}
Throughout this proof, \( c \) denotes a generic positive constant, independent of
\( h \), \( \tau \), $ R_h $, and $ R_{h,\tau} $, which may vary from one occurrence to another.
The remainder of the proof is organized into
five steps.

\textbf{Step 1.}
By definition, for any \( 0 \leqslant j < J \), we have almost surely  
\[
\begin{aligned}
  y_{h,R_h}({j\tau+\tau}) - y_{h,R_h}({j\tau})
  &= -\int_{{j\tau}}^{{j\tau+\tau}} \mathbbm{1}_{[0,t_{R_h}^{h,\varrho}]}(t)
  \bigl[ A_h y_{h,R_h}(t) + \mathcal{P}_h G(y_{h,R_h}(t)) \bigr] \, \mathrm{d}t \\
  &\quad + \mathcal{P}_h \int_{{j\tau}}^{{j\tau+\tau}} \mathbbm{1}_{[0,t_{R_h}^{h,\varrho}]}(t) F(y_{h,R_h}(t)) \, \mathrm{d}W_H(t).
\end{aligned}
\]  
For each \( 0 \leqslant j \leqslant J \), define \( E_{j,R_h} := Y_{j,R_h} - y_{h,R_h}({j\tau}) \), and decompose \( E_{j,R_h} = \xi_{j,R_h} + \eta_{j,R_h} \), where \( (\xi_{j,R_h})_{j=0}^J \) satisfies  
\[
\begin{cases}
  \xi_{j+1,R_h} - \xi_{j,R_h} = -\int_{{j\tau}}^{{j\tau+\tau}}
  \mathbbm{1}_{[0,t_{R_h}^{h,\varrho}]}(t) \, A_h\xi_{j+1,R_h} \, \mathrm{d}t \\
  \qquad\qquad\qquad\qquad + \int_{{j\tau}}^{{j\tau+\tau}}
  \mathbbm{1}_{[0,t_{R_h}^{h,\varrho}]}(t) \, A_h\bigl[ y_{h,R_h}(t) - y_{h,R_h}({j\tau+\tau})\bigr] \, \mathrm{d}t, & 0 \leqslant j < J, \\
  \xi_{0,R_h} = 0.
\end{cases}
\]  
It follows that, for any \( 0 \leqslant j < J \), almost surely,  
\[
\begin{aligned}
  \eta_{j+1,R_h} - \eta_{j,R_h}
  &= -\int_{{j\tau}}^{{j\tau+\tau}} \mathbbm{1}_{[0,t_{R_h}^{h,\varrho}]}(t) A_h\eta_{j+1,R_h}  \, \mathrm{d}t \\
  &\quad + \mathcal{P}_h\int_{{j\tau}}^{{j\tau+\tau}} \mathbbm{1}_{[0,t_{R_h}^{h,\varrho}]}(t) \bigl[
    G(y_{h,R_h}(t)) - G(Y_{j+1,R_h})
  \bigr] \, \mathrm{d}t \\
  &\quad + \mathcal{P}_h \int_{{j\tau}}^{{j\tau+\tau}}
  \mathbbm{1}_{[0,t_{R_h}^{h,\varrho}]}(t) \bigl[
    F(Y_{j,R_h}) - F(y_{h,R_h}(t))
  \bigr] \, \mathrm{d}W_H(t).
\end{aligned}
\]  
Taking the \( L^2(\Omega; \mathbb{L}^2) \)-inner product with \( \eta_{j+1,R_h} \), summing over \( j = 0, \dots, m \) for \( 0 \leqslant m < J \), and applying Hölder's inequality and Itô's isometry, we obtain  
\begin{small}
\[
\begin{aligned}
  \frac{1}{2} \mathbb{E} \|\eta_{m+1,R_h}\|_{\mathbb{L}^2}^2
  &\leqslant \mathbb{E} \bigg[
    \sum_{j=0}^m \int_{{j\tau}}^{{j\tau+\tau}} \mathbbm{1}_{[0,t_{R_h}^{h,\varrho}]}(t)
    \big\langle -A_h \eta_{j+1,R_h}, \, \eta_{j+1,R_h} \big\rangle \, \mathrm{d}t
  \bigg] \\
  &\quad + \mathbb{E} \bigg[
    \sum_{j=0}^m \int_{{j\tau}}^{{j\tau+\tau}} \mathbbm{1}_{[0,t_{R_h}^{h,\varrho}]}(t)
    \big\langle G(y_{h,R_h}(t)) - G(y_{h,R_h}({j\tau+\tau})), \, \eta_{j+1,R_h} \big\rangle \, \mathrm{d}t
  \bigg] \\
  &\quad + \mathbb{E} \bigg[
    \sum_{j=0}^m \int_{{j\tau}}^{{j\tau+\tau}} \mathbbm{1}_{[0,t_{R_h}^{h,\varrho}]}(t)
    \big\langle G(y_{h,R_h}({j\tau+\tau})) - G(Y_{j+1,R_h}), \, \eta_{j+1,R_h} \big\rangle \, \mathrm{d}t
  \bigg] \\
  &\quad + \frac{1}{2} \mathbb{E} \bigg[
    \sum_{j=0}^m \int_{{j\tau}}^{{j\tau+\tau}} \mathbbm{1}_{[0,t_{R_h}^{h,\varrho}]}(t)
    \|F(Y_{j,R_h}) - F(y_{h,R_h}(t))\|_{\gamma(H,\mathbb{L}^2)}^2 \, \mathrm{d}t
  \bigg] \\
  &=: I_m^{(1)} + I_m^{(2)} + I_m^{(3)} + I_m^{(4)}.
\end{aligned}
\]
\end{small}

    \textbf{Step 2.}
    Let us estimate \( I_m^{(1)} \) and \( I_m^{(2)} \).
    For \( I_m^{(1)} \), the definition of \( A_h \) yields
    \begin{align}
      \label{eq:Im1}
      I_m^{(1)} = - \mathbb E \left[\,
        \sum_{j=0}^m \int_{{j\tau}}^{{j\tau+\tau}}
        \mathbbm{1}_{[0,t_{R_h}^{h,\varrho}]}(t)
        \norm{\nabla \eta_{j+1,R_h}}_{\mathbb L^2}^2 \, \mathrm{d}t
      \,\right].
    \end{align}
    For \( I_m^{(2)} \), considering the definition of \( G \) in \cref{eq:G-def}, an application of integration by parts leads to the following identity for any \( 0 \leqslant j \leqslant m \) and \( t \in [{j\tau}, {j\tau+\tau}] \):
    \begin{align*}
      & \big\langle \eta_{j+1,R_h}, \, G(y_{h,R_h}(t)) - G(y_{h,R_h}({j\tau+\tau})) \big\rangle \\
      ={}& \left\langle \eta_{j+1,R_h}, \, \left[\big(y_{h,R_h}(t) - y_{h,R_h}({j\tau+\tau})\big) \cdot \nabla\right] \Big( y_{h,R_h}(t) - \frac12 y_{h,R_h}({j\tau+\tau}) \Big) \right\rangle \\
      &\quad {} + \left\langle \eta_{j+1,R_h}, \, \Big[\nabla \cdot \Big( \frac12y_{h,R_h}(t) - y_{h,R_h}({j\tau+\tau}) \Big)\Big] \Big( y_{h,R_h}(t) - y_{h,R_h}({j\tau+\tau}) \Big) \right\rangle \\
      &\quad {} - \frac{1}{2} \Big\langle \big[\big( y_{h,R_h}(t) - y_{h,R_h}({j\tau+\tau}) \big) \cdot \nabla\big] \eta_{j+1,R_h}, \, y_{h,R_h}({j\tau+\tau}) \Big\rangle \\
      &\quad {} - \Big\langle \big[y_{h,R_h}({j\tau+\tau})\cdot\nabla\big] \eta_{j+1,R_h}, \, y_{h,R_h}(t) - y_{h,R_h}({j\tau+\tau}) \Big\rangle.
    \end{align*}
    By Hölder's inequality, the \( h \)-independent embeddings (see \cref{rem:discrete_sobolev})
\[
\dot H_h^{\varrho,2} \hookrightarrow W^{1,2/(2-\varrho)}(\mathcal O;\mathbb R^2), \quad
\dot H_h^{\varrho,2} \hookrightarrow \mathbb L^\infty,
\quad\dot{H}_h^{1,2} \hookrightarrow \mathbb{L}^q \text{ for all } q \in (2,\infty), 
\]  
and the definition of \( t_{R_h}^{h,\varrho} \) in \cref{eq:tR-new}, we derive, for any \( 0 \leqslant j < J \) and \( t \in [j\tau, j\tau+\tau] \),  
\begin{equation}
  \label{eq:57}
\begin{aligned}
  & \mathbbm{1}_{[0,t_{R_h}^{h,\varrho}]}(t) \,
  \big\langle \eta_{j+1,R_h}, \, G(y_{h,R_h}(t)) - G(y_{h,R_h}({j\tau+\tau})) \big\rangle \\
  \leqslant{} & cR_h \mathbbm{1}_{[0,t_{R_h}^{h,\varrho}]}(t) \, \norm{y_{h,R_h}(t) - y_{h,R_h}({j\tau+\tau})}_{\mathbb{L}^2}
   \norm{\nabla \eta_{j+1,R_h}}_{\mathbb{L}^2},
   \quad\text{$\mathrm{d}\mathbb P\otimes\mathrm{d}t$-a.e.}
\end{aligned}
\end{equation}
    Consequently,
    \begin{small}
    \begin{align}
      \label{eq:Im2}
      I_m^{(2)} \leqslant \mathbb E
      \left[\,
        \sum_{j=0}^m \int_{{j\tau}}^{{j\tau+\tau}} \mathbbm{1}_{[0,t_{R_h}^{h,\varrho}]}(t)
        \Big(
          \frac14 \norm{\nabla \eta_{j+1,R_h}}_{\mathbb{L}^2}^2
          + c R_h^2 \norm{y_{h,R_h}(t) - y_{h,R_h}({j\tau+\tau})}_{\mathbb{L}^2}^2
        \Big) \, \mathrm{d}t
      \,\right].
    \end{align}
  \end{small}


    \textbf{Step 3.} We now proceed to estimate \( I_m^{(3)} \) and \( I_m^{(4)} \).
    Firstly, observe that the definition of \( t_{R_h}^{h,\varrho} \) in \cref{eq:tR-new} implies that, almost surely,
    \[
      \Big\lVert \mathbbm{1}_{[0,t_{R_h}^{h,\varrho}]} A_h\big[ y_{h,R_h} - y_{h,R_h}(j\tau+\tau)\big] \Big\rVert_{C((j\tau,j\tau+\tau]; \dot{H}_h^{\varrho-2,2})} \leqslant 2R_h
      \quad\text{for all $0\leqslant j < J$,}
    \]
    which allows us to apply \cref{lem:auxi} to deduce that
    \begin{align}
      \max_{1 \leqslant j \leqslant J}
      \norm{\xi_{j,R_h}}_{\dot{H}_h^{(1+\varrho)/2,2}} \leqslant c R_h,
      \quad \text{almost surely.}
      \label{eq:xi-regu1}
    \end{align}
    For any \( 0 \leqslant j \leqslant m \) and \( t \in [{j\tau}, {j\tau+\tau}] \), a direct calculation
    using integration by parts yields
    \begin{align*}
      & \mathbbm{1}_{[0,t_{R_h}^{h,\varrho}]}(t)\,
      \Big\langle  G(y_{h,R_h}({j\tau+\tau})) - G(Y_{j+1,R_h}), \, \eta_{j+1,R_h} \Big\rangle \\
      ={}& \mathbbm{1}_{[0,t_{R_h}^{h,\varrho}]}(t) \, \biggl\{
        - \frac{1}{2} \Big\langle \eta_{j+1,R_h}, \, (E_{j+1,R_h} \cdot \nabla) \big( y_{h,R_h}({j\tau+\tau}) + \xi_{j+1,R_h} \big) \Big\rangle \\
         & \qquad\qquad\qquad {} + \frac{1}{2} \Big\langle (E_{j+1,R_h} \cdot \nabla) \eta_{j+1,R_h}, \, y_{h,R_h}({j\tau+\tau}) + \xi_{j+1,R_h} \Big\rangle \\
         & \qquad\qquad\qquad {} + \frac{1}{2} \Big\langle \big(\nabla \cdot y_{h,R_h}({j\tau+\tau})\big) \xi_{j+1,R_h}, \, \eta_{j+1,R_h} \Big\rangle \\
         & \qquad\qquad\qquad {} + \Big\langle \big(y_{h,R_h}({j\tau+\tau}) \cdot \nabla\big) \eta_{j+1,R_h}, \, \xi_{j+1,R_h} \Big\rangle
       \biggr\}.
      \end{align*}
      Thus, following the argument in \cref{eq:57} and utilizing the stability estimate \cref{eq:xi-regu1} along with the
      equality $ E_{j+1,R_h} = \xi_{j+1,R_h} + \eta_{j+1,R_h} $, we infer that
      \begin{align*}
        & \mathbbm{1}_{[0,t_{R_h}^{h,\varrho}]}(t) \, \Big\langle G(y_{h,R_h}({j\tau+\tau})) - G(Y_{j+1,R_h}), \, \eta_{j+1,R_h} \Big\rangle \\
        \leqslant{}
        & c R_h \mathbbm{1}_{[0,t_{R_h}^{h,\varrho}]}(t) \, \|\nabla \eta_{j+1,R_h}\|_{\mathbb{L}^2}
        \big(
          \|\xi_{j+1,R_h}\|_{\mathbb{L}^2} + \|\eta_{j+1,R_h}\|_{\mathbb{L}^2}
        \big),
        \quad\text{$\mathrm{d}\mathbb P\otimes\mathrm{d}t$-a.e.}
      \end{align*}
      Consequently,
      \begin{small}
      \begin{equation}
        \label{eq:Im3}
        I_m^{(3)} \leqslant \mathbb{E}
        \left[\,
          \sum_{j=0}^m \int_{{j\tau}}^{{j\tau+\tau}} \mathbbm{1}_{[0,t_{R_h}^{h,\varrho}]}(t) \Big(
            \frac14 \|\nabla \eta_{j+1,R_h}\|_{\mathbb{L}^2}^2
            + c R_h^2 \|\xi_{j+1,R_h}\|_{\mathbb{L}^2}^2
            + cR_h^2 \|\eta_{j+1,R_h}\|_{\mathbb{L}^2}^2
          \Big) \, \mathrm{d}t
        \,\right].
      \end{equation}
    \end{small}
      \hskip -.2em For \( I_m^{(4)} \), we apply the Lipschitz continuity of $ F $
      from \cref{lem:F}(i) and the identity
      \[
        \mathbbm{1}_{[0,t_{R_h}^{h,\varrho}]}(t) \bigl( Y_{j,R_h} - y_{h,R_h}(t) \bigr) =
        \mathbbm{1}_{[0,t_{R_h}^{h,\varrho}]}(t) \bigl( \xi_{j,R_h} + \eta_{j,R_h} + y_{h,R_h}({j\tau}) - y_{h,R_h}(t) \bigr),
      \]
      for all \( 0 \leqslant j < J \) and \( t \in [{j\tau}, {j\tau+\tau}] \), to derive the following estimate:
      \begin{small}
      \begin{equation}
        \label{eq:Im4}
        I_m^{(4)} \leqslant c \mathbb{E} \left[\,
          \sum_{j=0}^m \int_{{j\tau}}^{{j\tau+\tau}} \mathbbm{1}_{[0,t_{R_h}^{h,\varrho}]}(t)
          \big( \norm{y_{h,R_h}(t) - y_{h,R_h}({j\tau})}_{\mathbb{L}^2}^2 + \norm{\xi_{j,R_h}}_{\mathbb{L}^2}^2 + \norm{\eta_{j,R_h}}_{\mathbb{L}^2}^2 \big) \, \mathrm{d}t
        \,\right].
      \end{equation}
    \end{small}

\textbf{Step 4.}
For any \( p \in (2/(\varrho-1),\infty) \), applying \cref{lem:auxi}, we obtain
\begin{small}
\begin{align*}
\mathbb{E}\left[\max_{1 \leqslant j \leqslant J} \|\xi_{j,R_h}\|_{\mathbb{L}^2}^2\right]
\leqslant
c \mathbb{E} \left[
\left( \sum_{j=0}^{J-1} \int_{{j\tau}}^{{j\tau+\tau}}
\mathbbm{1}_{[0,t_{R_h}^{h,\varrho}]}(t)
\|y_{h,R_h}(t)-y_{h,R_h}({j\tau+\tau})\|_{\dot{H}_h^{\varrho-1,2}}^p \, \mathrm{d}t
\right)^{2/p}
\right].
\end{align*}
\end{small}
\hskip -.3em In conjunction with Lemma \ref{lem:yh-interp}, this inequality leads us to
\begin{equation}
\mathbb{E}\Big[ \, \max_{1 \leqslant j \leqslant J} \|\xi_{j,R_h}\|_{\mathbb{L}^2}^2 \, \Big]
\leqslant c\tau R_h^2(1 + \tau^{(5-3\varrho)/2} R_h^2).
\label{eq:xi-conv}
\end{equation}
Furthermore, Lemma \ref{lem:yh-interp} also implies that
\begin{equation}
\label{eq:yh-interp-L2}
\begin{split}
& \mathbb{E} \bigg[\,
\sum_{j=0}^{J-1} \int_{{j\tau}}^{{j\tau+\tau}}
\mathbbm{1}_{[0, t_{R_h}^{h,\varrho}]}(t)
\Big( \|y_{h,R_h}(t) - y_{h,R_h}({j\tau+\tau})\|_{\mathbb{L}^2}^2 \\
& \qquad\qquad {} + \|y_{h,R_h}(t) - y_{h,R_h}(t_{j})\|_{\mathbb{L}^2}^2 \Big)
\, \mathrm{d}t
\,\bigg] \leqslant c \tau R_h^2 \left( 1 + \tau R_h^2 \right).
\end{split}
\end{equation}
By combining the estimates for \( I_m^{(1)}, \ldots, I_m^{(4)} \)
from \cref{eq:Im1,eq:Im2,eq:Im3,eq:Im4} and
utilizing inequalities \cref{eq:xi-conv,eq:yh-interp-L2},
we infer that there exists a constant \( c_0 > 0 \),
independent of \( h \), \( \tau \), and $ R_h $, such that for any \( 0 \leqslant m < J \),
\begin{equation}
\label{eq:etaR1}
\begin{aligned}
& \mathbb{E} \left[\,
\|\eta_{m+1,R_h}\|_{\mathbb{L}^2}^2
+ \sum_{j=0}^m \int_{{j\tau}}^{{j\tau+\tau}} \mathbbm{1}_{[0,t_{R_h}^{h,\varrho}]}(t) \|\nabla \eta_{j+1,R_h}\|_{\mathbb{L}^2}^2 \, \mathrm{d}t
\,\right] \\
\leqslant{} &
c_0 R_h^2 \, \mathbb{E} \bigg[\,
\sum_{j=0}^m \int_{{j\tau}}^{{j\tau+\tau}} \mathbbm{1}_{[0,t_{R_h}^{h,\varrho}]}(t)
\|\eta_{j+1,R_h}\|_{\mathbb{L}^2}^2 \, \mathrm{d}t \, \bigg]
+ c_0 \tau R_h^4\big(1 + \tau^{(5-3\varrho)/2}R_h^2\big).
\end{aligned}
\end{equation}
It follows that
\begin{equation}
\label{eq:conv1}
\begin{aligned}
& \max_{1 \leqslant j \leqslant J} \|\eta_{j,R_h}\|_{L^2(\Omega; \mathbb{L}^2)}^2
+ \mathbb{E} \biggl[\,
\sum_{j=0}^{J-1} \int_{{j\tau}}^{{j\tau+\tau}} \mathbbm{1}_{[0,t_{R_h}^{h,\varrho}]}(t)
\|\nabla \eta_{j+1,R_h}\|_{\mathbb{L}^2}^2 \, \mathrm{d}t
\,\biggr] \\
\leqslant & \, c \tau \exp(c R_h^2).
\end{aligned}
\end{equation}
To prove this, consider two cases. If \( c_0R_h^2 \tau \leqslant 1/2 \),
then applying the discrete Gronwall's inequality to \cref{eq:etaR1}
yields \cref{eq:conv1}. For the alternative case where \( c_0R_h^2 \tau > 1/2 \),
a routine argument using the property $ \dual{G(u_h),u_h} = 0 $ for all $ u_h \in \mathbb L_{h,\mathrm{div}} $
ensures that $ \max_{0 \leqslant j \leqslant J} \|Y_{j,R_h}\|_{L^2(\Omega;\mathbb L^2)} $
is uniformly bounded with respect to \( h \), \( \tau \), and $ R_h $.
This result, combined with \cref{eq:xi-regu1} and the almost sure bound
\[
  \big\lVert\mathbbm{1}_{[0,t_{R_h}^{h,\varrho}]} y_{h,R_h}\big\rVert_{C((0,T];\dot H_h^{\varrho,2})}
  \leqslant R_h,
\]
allows us to deduce from the equality \( \eta_{j,R_h} = Y_{j,R_h} - y_{h,R_h}(j\tau) - \xi_{j,R_h} \)
for \( 1 \leqslant j \leqslant J \) that
\[
    \mathbb E\bigg[\,\sum_{j=0}^{J-1} \int_{j\tau}^{j\tau+\tau}
    \mathbbm{1}_{[0,t_{R_h}^{h,\varrho}]}(t) \|\eta_{j+1,R_h}\|_{\mathbb L^2}^2
    \, \mathrm{d}t \,\bigg] \leqslant cR_h^2.
\]
Consequently, in view of \cref{eq:etaR1},
we assert that \cref{eq:conv1} holds in this case as well.

\textbf{Step 5.}
We define the sequences $(\chi_{j,R_h}^{(i)})_{j=0}^J$ for $i = 1,2,3,4$ as follows:
Given $0 \leqslant j < J$, with the initial conditions $\chi_{0,R_h}^{(i)} = 0$ for all $i = 1,2,3,4$,
the recursion relations are given by
\begin{small}
\[
\begin{aligned}
& \chi_{j+1,R_h}^{(1)} - \chi_{j,R_h}^{(1)} + \tau A_h \chi_{j+1,R_h}^{(1)} = \int_{j\tau}^{j\tau+\tau} \mathbbm{1}_{[0,t_{R_h}^{h,\varrho}]}(t) \, A_h \big[ y_{h,R_h}(t) - y_{h,R_h}(j\tau+\tau) \big] \, \mathrm{d}t, \\
& \chi_{j+1,R_h}^{(2)} - \chi_{j,R_h}^{(2)} + \tau A_h \chi_{j+1,R_h}^{(2)} = \mathcal{P}_h \int_{j\tau}^{j\tau+\tau} \mathbbm{1}_{[0,t_{R_h}^{h,\varrho}]}(t) \, \big[ G(y_{h,R_h}(t)) - G(y_{h,R_h}(j\tau+\tau)) \big] \, \mathrm{d}t, \\
& \chi_{j+1,R_h}^{(3)} - \chi_{j,R_h}^{(3)} + \tau A_h \chi_{j+1,R_h}^{(3)} = \mathcal{P}_h \int_{j\tau}^{j\tau+\tau} \mathbbm{1}_{[0,t_{R_h}^{h,\varrho}]}(t) \, \big[ G(y_{h,R_h}(j\tau+\tau)) - G(Y_{j+1,R_h}) \big] \, \mathrm{d}t, \\
& \chi_{j+1,R_h}^{(4)} - \chi_{j,R_h}^{(4)} + \tau A_h \chi_{j+1,R_h}^{(4)} = \mathcal{P}_h \int_{j\tau}^{j\tau+\tau} \mathbbm{1}_{[0,t_{R_h}^{h,\varrho}]}(t) \, \big[ F(Y_{j,R_h}) - F(y_{h,R_h}(t)) \big] \, \mathrm{d}W_H(t).
\end{aligned}
\]
\end{small}
\hskip -.2em For the sequence \((\chi_{j,R_h}^{(1)})_{j=1}^J\), an argument similar to the derivation of inequality \eqref{eq:xi-conv} yields the bound:
\begin{equation}
  \label{eq:E1}
\mathbb{E}\left[
\max_{1 \leqslant j \leqslant J} \lVert \chi_{j,R_h}^{(1)}\rVert_{\mathbb{L}^2}^2
\right] \leqslant c\tau R_h^2(1+\tau^{(5-3\varrho)/2} R_h^2).
\end{equation}
For the sequence \((\chi_{j,R_h}^{(2)})_{j=1}^J\), note that
using techniques similar to those in Step 2 yields that,
for any $ 0 \leqslant j < J $ and $ t \in [j\tau, j\tau+\tau] $,
\begin{align*}
  & \mathbbm{1}_{[0,t_{R_h}^{h,\varrho}]}(t)
  \left\lVert\mathcal P_h\big[G(y_{h,R_h}(t)) - G(y_{h,R_h}({j\tau+\tau}))\big] \right\rVert_{\dot H_h^{-1,2}} \\
  \leqslant{}& cR_h \mathbbm{1}_{[0,t_{R_h}^{h,\varrho}]}(t)
  \norm{y_{h,R_h}(t) - y_{h,R_h}({j\tau+\tau})}_{\mathbb L^2},
  \quad\text{$\mathrm{d}\mathbb P\otimes\mathrm{d}t$-a.e.,}
\end{align*}
This enables us to use \cref{lem:auxi} to deduce the following bound:
\begin{align}
\mathbb{E} \left[
\max_{1 \leqslant j \leqslant J} \norm{\chi_{j,R_h}^{(2)}}_{\mathbb{L}^2}^2
\right]
\leqslant cR_h^2 \mathbb{E} \left[
\sum_{j=0}^{J-1} \int_{{j\tau}}^{{j\tau+\tau}}
\mathbbm{1}_{[0,t_{R_h}^{h,\varrho}]}(t)
\norm{y_{h,R_h}(t) - y_{h,R_h}({j\tau+\tau})}_{\mathbb L^2}^2 \, \mathrm{d}t
\right].
\label{eq:E2}
\end{align}
For the sequence \((\chi_{j,R_h}^{(3)})_{j=1}^J\), we begin by noting that,
through the application of techniques analogous to those in Step 2,
coupled with the stability property of \( (\xi_{j,R_h})_{j=1}^J \) as stated in \cref{eq:xi-regu1},
the equality \( Y_j = Y_{j,R_h} \) for each \( 0 \leqslant j \leqslant J \) almost surely
on \( \Omega_{\varrho,R_h,R_{h,\tau}}^{h,\tau} \),
and the standard inequality (cf.~\cite[Lemma~4.9.2]{Brenner2008})
\[
\norm{u_h}_{\mathbb L^\infty} \leqslant c\sqrt{\ln\left(\frac{1}{h}\right)} \, \norm{\nabla u_h}_{\mathbb L^2},
\quad \forall u_h \in \mathbb L_{h,\text{div}},
\]
we deduce that, for any \( 1 \leqslant j \leqslant J \),
\begin{align*}
 \left\lVert\mathcal{P}_h\big[G(y_{h,R_h}(t_{j})) - G(Y_{j,R_h})\big]\right\rVert_{\dot{H}_h^{-1,2}} 
\leqslant
 cR_h \big(
\norm{\xi_{j,R_h}}_{\mathbb{L}^2}
+ \norm{\eta_{j,R_h}}_{\mathbb{L}^2}
\big) \\
+ c\sqrt{\ln\left(\frac{1}{h}\right)} \, R_{h,\tau} \norm{\nabla \eta_{j,R_h}}_{\mathbb{L}^2}
\quad \text{almost surely on } \Omega_{\varrho,R_h,R_{h,\tau}}^{h,\tau}.
\end{align*}
Subsequently, invoking \cref{lem:auxi}, we derive the following bound:
\begin{equation}
  \label{eq:E3}
  \begin{aligned}
& \mathbb{E} \left[
\mathbbm{1}_{\Omega_{\varrho,R_h,R_{h,\tau}}^{h,\tau}}
\max_{1 \leqslant j \leqslant J} \big\lVert \chi_{j,R_h}^{(3)} \big\rVert_{\mathbb{L}^2}^2
\right] \\
\leqslant{}
& c \mathbb{E} \Biggl[
\mathbbm{1}_{\Omega_{\varrho,R_h,R_{h,\tau}}^{h,\tau}}
\sum_{j=0}^{J-1} \int_{{j\tau}}^{{j\tau+\tau}}
\Bigl(
R_h^2\norm{\eta_{j+1,R_h}}_{\mathbb{L}^2}^2 + R_h^2\norm{\xi_{j+1,R_h}}_{\mathbb{L}^2}^2 \\
& \qquad\qquad\qquad\qquad\qquad\qquad\qquad
{} + \ln\left(\frac{1}{h}\right)R_{h,\tau}^2\norm{\nabla\eta_{j+1,R_h}}_{\mathbb{L}^2}^2
\Bigr) \, \mathrm{d}t
\Biggr].
\end{aligned}
\end{equation}
For the sequence \((\chi_{j,R_h}^{(4)})_{j=1}^J\),
using Lemma \ref{lem:DSMLP} and the Lipschitz property of \(F\) from \cref{lem:F}(i),
we obtain the bound
\begin{align}
\mathbb{E} \left[\max_{1 \leqslant j \leqslant J}
\big\lVert \chi_{j,R_h}^{(4)} \big\rVert_{\mathbb{L}^2}^2\right]
\leqslant c \mathbb{E} \left[\sum_{j=0}^{J-1} \int_{{j\tau}}^{{j\tau+\tau}}
\mathbbm{1}_{[0,t_{R_h}^{h,\varrho}]}(t) \norm{y_{h,R_h}(t) - Y_{j,R_h}}_{\mathbb{L}^2}^2 \, \mathrm{d}t\right].
\label{eq:E4}
\end{align}
By combining the preceding bounds in \cref{eq:E1,eq:E2,eq:E3,eq:E4} and
utilizing estimates \cref{eq:xi-conv,eq:yh-interp-L2,eq:conv1}, we arrive at
\begin{align*}
& \mathbb{E} \left[
\mathbbm{1}_{\Omega_{\varrho,R_h,R_{h,\tau}}^{h,\tau}}
\max_{1 \leqslant j \leqslant J} \normb{
\chi_{j,R_h}^{(1)} + \chi_{j,R_h}^{(2)} + \chi_{j,R_h}^{(3)} + \chi_{j,R_h}^{(4)}
}_{\mathbb{L}^2}^2
\right] \\
\leqslant{}
& c\tau\exp(cR_h^2) \left(1 + \ln\left(\frac{1}{h}\right) R_{h,\tau}^2\right).
\end{align*}
The desired pathwise uniform convergence in \eqref{eq:local-conv} follows from the
observation that almost surely on \(\Omega_{\varrho,R_h,R_{h,\tau}}^{h,\tau}\),
\[
Y_j - y_h({j\tau}) = \chi_{j,R_h}^{(1)} + \chi_{j,R_h}^{(2)} + \chi_{j,R_h}^{(3)} + \chi_{j,R_h}^{(4)}
\quad \text{for all } 1 \leqslant j \leqslant J.
\]
This completes the proof of Theorem \ref{thm:yhR-YR}.
\end{proof}

\begin{remark}
  \label{rem:discrete_sobolev}
  The above proof relies on standard discrete Sobolev embeddings.
  For completeness, we provide key details. Let \(c > 0\)
  denote a generic \(h\)-independent constant.
  For \(\varrho \in (1, 2)\) and \(u_h \in \mathbb{L}_{h,\mathrm{div}}\), we have:
  \begin{align*}
    \norm{A_2^{-1}\mathcal{P} A_h u_h}_{\dot{H}^{\varrho,2}}
    &= \norm{\mathcal{P} A_h u_h}_{\dot{H}^{\varrho-2,2}} 
    = \sup_{\varphi \in \dot H^{2-\varrho,2}} \frac{\dual{A_hu_h,\mathcal P_h\varphi}}{\norm{\varphi}_{\dot{H}^{2-\varrho,2}}} \\
    &\leqslant c \norm{A_h u_h}_{\dot{H}_h^{\varrho-2,2}} \quad \text{(by Lemma~\ref{lem:Ph}(ii))} \\
    &= c \norm{u_h}_{\dot{H}_h^{\varrho,2}}.
  \end{align*}
  By \cite[Theorem~3.1]{Girault2003}, there exists \( w_h \in \mathbb L_{h,\mathrm{div}} \) satisfying
  \begin{align*}
    & \norm{A_2^{-1}\mathcal PA_hu_h - w_h}_{\mathbb L^2} \leqslant ch^\varrho \norm{u_h}_{\dot H_h^{\varrho,2}}, \\
    & \norm{w_h}_{\mathbb L^\infty} \leqslant c\norm{u_h}_{\dot H_h^{\varrho,2}}.
  \end{align*}
  Applying \cite[Theorem 3.1]{Stenberg1990} yields
  \[
    \norm{A_2^{-1}\mathcal{P} A_h u_h - u_h}_{\mathbb{L}^2} \leqslant ch^2 \norm{u_h}_{\dot{H}_h^{2,2}}
    \leqslant ch^\varrho \norm{u_h}_{\dot H_h^{\varrho,2}}
    \quad\mathrm{(by \cref{eq:inverse})}.
  \]
  Combining these estimates with the inverse estimate
  \( \norm{u_h-w_h}_{\mathbb L^\infty} \leqslant ch^{-1}\norm{u_h-w_h}_{\mathbb L^2} \),
  we obtain
  \( \norm{u_h}_{\mathbb{L}^\infty} \leqslant c \norm{u_h}_{\dot{H}_h^{\varrho,2}} \).
  Similarly, we have
  \[
    \norm{\nabla u_h}_{\mathbb L^{2/(2-\varrho)}}
    \leqslant c \norm{u_h}_{\dot H_h^{\varrho,2}}.
  \]
\end{remark}

Finally, using the established convergence results presented in Theorems \ref{thm:yR-yhR} and \ref{thm:yhR-YR}, along with the derived regularity result of the spatial semidiscretization \cref{eq:yh} as stated in \cref{thm:yh-regu}, we are able to derive the rate of convergence in probability for the solution of the full discretization \cref{eq:Y} to the global mild solution of the model problem \cref{eq:model}. The main result is presented in the following theorem.

\begin{theorem}
  Suppose \( y_0 \in L_{\mathcal F_0}^4(\Omega;\dot H^{3/2,2}) \).
  Fix \( \alpha \in (2,3) \) and \( \beta \in (0,1) \).
  Let \( y \) be the global mild solution to \cref{eq:model} in \( V_{\varrho} \) with
  \( \varrho \in (\frac{\alpha}{2}, \frac{3}{2}) \).
  Let $ (Y_j)_{j=0}^J $ be the solution to the full discretization \cref{eq:Y}.
  Then, for any \( \epsilon > 0 \),
  \begin{equation}
    \label{eq:y-Y}
    \lim_{\substack{h \to 0, \, \tau \to 0 \\ \tau \leqslant h}} \mathbb{P}\left\{
    \frac{\max_{1 \leqslant j \leqslant J} \norm{y(j\tau) - Y_j}_{\mathbb{L}^2}^2}{h^\alpha + \tau^\beta} \geqslant \epsilon \right\} = 0.
  \end{equation}
\end{theorem}
\begin{proof}
Without loss of generality, we may assume that \( h \) and \( \tau \) are sufficiently small. Let \( c \) denote a generic positive constant, independent of the discretization parameters \( h \) and \( \tau \), whose value may vary between occurrences.
 Select any \( \gamma_0 \in (0, 2\varrho - \alpha) \) and \( \gamma_1 \in (0, 1 - \beta) \). Define
\[
R = \sqrt{\frac{\gamma_0}{c_4} \ln \frac{1}{h}},
 \quad R_h = \sqrt{\frac{\gamma_1}{c_6} \ln \frac{1}{\tau}},
  \quad R_{h,\tau} = \sqrt{
    \frac{\tau^{\beta+\gamma_1 - 1}}{
      c_5\big(\ln\frac1h\big)\ln\big(1+\ln\frac1\tau\big)
    }
  },
\]
where \( c_4 \), \(c_5\), and \( c_6 \) will be specified later in the proof.
By invoking Lemma 3.1 from \cite{Prohl2013}, which states that
\( \mathbb{E}[\max_{1\leqslant j\leqslant J} \|Y_j\|_{\mathbb{L}^2}^4] \) is uniformly bounded with respect to \( h \) and \( \tau \), we obtain the probability estimate
\[
\mathbb{P}\left( \max_{1 \leqslant j \leqslant J} \|Y_j\|_{\mathbb{L}^2} > R_{h,\tau} \right) \leqslant \frac{c}{R_{h,\tau}^4}.
\]
Using Theorem \ref{thm:yh-regu}, we have
\[
\mathbb{P}\left( \|y_h\|_{C([0,T];\dot{H}_h^{\varrho,2})} > R_h \right) \leqslant \frac{c}{\ln(1 + R_h)}.
\]
Thus, for the set \( \Omega_{\varrho,R_h,R_{h,\tau}}^{h,\tau} \) defined in \cref{eq:local_sample_set}, we have
\[
\mathbb{P}\left( \Omega_{\varrho,R_h,R_{h,\tau}}^{h,\tau} \right) \geqslant 1 - \frac{c}{\ln(1 + R_h)} - \frac{c}{R_{h,\tau}^4}.
\]
By combining this with \cref{eq:P-tRrho}, we obtain
\[
\mathbb{P}\left( \Omega_{\varrho,R_h,R_{h,\tau}}^{h,\tau} \cap \{t_{R,\varrho} = T\} \right) \geqslant 1 - \frac{c}{\ln(1 + R)} - \frac{c}{\ln(1 + R_h)} - \frac{c}{R_{h,\tau}^4},
\]
where \( t_{R,\varrho} \) is as defined in \cref{eq:tR}.
Integrating this inequality with the error bounds (\ref{eq:yR-yhR}) and (\ref{eq:local-conv}), we establish the existence of constants \( c_0, \ldots, c_6 \), independent of \( h \), \( \tau \), and \( \epsilon \), such that
\[
\begin{aligned}
& \mathbb{P}\left\{ \frac{\max_{1 \leqslant j \leqslant J} \|y(j\tau) - Y_j\|_{\mathbb{L}^2}^2}{h^\alpha + \tau^\beta} \geqslant \epsilon \right\} \\
& \quad \leqslant \frac{c_0}{\ln(1 + R)} + \frac{c_1}{\ln(1 + R_h)} + \frac{c_2}{R_{h,\tau}^4} \\
& \qquad + \frac{c_3 h^{2\varrho-\alpha} (\ln \frac{1}{h})^2 \exp(c_4 R^2) + c_5 \tau^{1-\beta} \exp(c_6 R_h^2) (1 + \ln(\frac{1}{h}) R_{h,\tau}^2)}{\epsilon}.
\end{aligned}
\]
Therefore, the desired convergence result \cref{eq:y-Y} follows from the observation that the right-hand side of the preceding inequality converges to zero as \( \tau \leqslant h \to 0 \). This concludes the proof.
\end{proof}

\begin{remark}
  \label{rem:pressure-conv}
  The theoretical findings of this study are directly applicable to the SNSEs of the following form:
  \begin{equation*}
    \begin{cases}
      \begin{aligned}
        \mathrm{d}z(t,x) &= \Bigl[ \Delta z(t,x) - (z(t,x) \cdot \nabla)z(t,x) - \nabla\psi(t,x) \Bigr]\,\mathrm{d}t \\
                         &\quad + \sum_{n\in\mathbb{N}} f_n(x,z(t,x)) \, \mathrm{d}\beta_n(t), && t\in[0,T], \, x \in \mathcal{O}, \\
        \nabla \cdot z(t,x) &= 0, && t\in[0,T], \, x \in \mathcal{O}, \\
        z(t,x) &= 0, && t \in [0,T], \, x \in \partial\mathcal{O}, \\
        z(0,x) &= z_0(x), && x \in \mathcal{O}.
      \end{aligned}
    \end{cases}
  \end{equation*}
  In the context of global mild solutions, it is clear that the process \( y \) from the model
  problem \cref{eq:model} is identical to the above process \( z \), provided that \( y_0 = z_0 \).
  Moreover, the pressure term \( \psi \) can be represented as:
  \[
    \psi(t,x) = \varphi(t,x) + \frac{\mathrm{d}}{\mathrm{d}t} \sum_{n\in\mathbb N}\int_0^{t}
    \widetilde\varphi_n(s,x) \, \mathrm{d}\beta_n(s),
    \quad t \in (0,T), \, x \in \mathcal O,
  \]
  where \( \varphi \) and \( \widetilde\varphi_n \), \( n \in \mathbb N \), are the pressure-related processes in \cref{eq:model},
  and \( \frac{\mathrm{d}}{\mathrm{d}t} \) denotes the generalized
  first-order time derivative.
  Using the convergence estimates from \cref{thm:yR-yhR,thm:yhR-YR},
  one can establish pathwise uniform convergence in \( \mathbb L^2 \) for the process
  \( \varphi \) with respect to probability, similar to the convergence shown in \cref{eq:y-Y}.
  However, such convergence is not achievable for the process \( \psi \) due to its limited temporal regularity.
\end{remark}

      \subsection{Some technical estimates}
      \label{ssec:technical_esti}

      This subsection presents some technical estimates that are instrumental in the proof of Theorem \ref{thm:yhR-YR}.
\begin{lemma}
    \label{lem:auxi}
    Let \( 0 \leqslant j_0 \leqslant J-1 \) and \( \varepsilon \in (0,1] \). For any \( \alpha \in [0,2) \) and \( p \in \left(\frac{2}{2-\alpha}, \infty\right) \) (or \( \alpha=1 \) and \( p=2 \)), there exists a constant \( c \), independent of \( h \), \( \tau \), \( j_0 \), \( \varepsilon \), and \( g_h \), such that
    \[
        \max_{1 \leqslant j \leqslant j_0+1} \|Z_j\|_{\mathbb{L}^2}
        \leqslant c \|g_h\|_{L^p(0,T;\dot{H}_h^{-\alpha,2})}
    \]
    holds for all \( g_h \in L^p(0,T;\dot{H}_h^{-\alpha,2}) \), where the sequence \( (Z_j)_{j=0}^{j_0+1} \) is defined by
    \[
        \begin{cases}
            Z_0 = 0, \\
            Z_{j+1} - Z_j + \tau A_h Z_{j+1} = \int_{j\tau}^{j\tau+\tau} g_h(t) \, \mathrm{d}t, & 0 \leqslant j < j_0, \\
            Z_{j_0+1} - Z_{j_0} + \varepsilon \tau A_h Z_{j_0+1} = \int_{j_0\tau}^{(j_0+\varepsilon)\tau} g_h(t) \, \mathrm{d}t.
        \end{cases}
    \]
\end{lemma}

      \begin{lemma}
        \label{lem:DSMLP}
        Suppose $g_h \in L_\mathbb F^2(\Omega\times(0,T);\gamma(H,\dot H_h^{0,2}))$.
       Define the sequence $(Z_j)_{j=0}^J$ by
        \[
          Z_{j+1} - Z_j + \tau A_h Z_j = \int_{{j\tau}}^{{j\tau+\tau}} g_h(t) \, \mathrm{d}W_H(t),
          \quad 0 \leqslant j < J,
        \]
        with $Z_0 = 0$. Then, there exists a constant $c > 0$, independent of $ g_h $,
        $h$, and $\tau$, such that
        \[
          \biggl(
            \mathbb{E} \max_{1 \leqslant j \leqslant J} \|Z_j\|_{\mathbb{L}^2}^2
          \biggr)^{1/2}
          \leqslant c\|g_h\|_{L^2(\Omega\times(0,T);\gamma(H,\dot H_h^{0,2}))}.
        \]
      \end{lemma}

      \begin{remark}
        The stability estimates in \cref{lem:auxi,lem:DSMLP}
        are standard. \cref{lem:auxi} follows from Theorem~3.2 in \cite[Chapter 2]{Ashyralyev2012},
        combined with Propositions 1.3, 1.4, and 6.2, and Theorem~4.36 in \cite{Lunardi2018}.
        Alternatively, it can be derived directly via spectral decomposition of $A_h$.
        The stability estimate in \cref{lem:DSMLP} can be found in
        \cite[Theorem~2.6]{Gyongy2007} and \cite[Proposition~5.4]{Neerven2022}.
      \end{remark}

\begin{lemma}
  \label{lem:yh-interp}
  Under the hypotheses of Theorem \ref{thm:yhR-YR},
  let $ t_{R_h}^{h,\varrho} $ be the stopping time defined in \cref{eq:tR-new},
  and let $ y_{h,R_h} $ be given by \cref{eq:yhR1}.
  For any $ \alpha \in [0, \varrho - 1] $, \( p \in [2,\infty) \),
  and $ R_h > 1 $, we have
  \begin{equation}
    \label{eq:yh-interp}
    \begin{split}
    & \mathbb{E}  \Biggl[ \,
      \sum_{j=0}^{J-1} \int_{{j\tau}}^{{j\tau+\tau}}
      \norm{y_{h,R_h}(t) - y_{h,R_h}(j\tau)}_{\dot{H}_h^{\alpha,2}}^p
      + \norm{y_{h,R_h}(t) - y_{h,R_h}({j\tau+\tau})}_{\dot{H}_h^{\alpha,2}}^p
      \, \mathrm{d}t
    \,\Biggr] \\
    \leqslant{}
    & c \tau^{\frac{p}2} R_h^p \left( 1 + \tau^{\frac{2 - 3\alpha}4p} R_h^p \right),
    \end{split}
  \end{equation}
  where the constant $ c $ is independent of $ h $, $ \tau $, and $ R_h $.
\end{lemma}
\begin{proof}
We decompose \( y_{h,R_h}(t) \) almost surely as:
\[
y_{h,R_h}(t) = I_1(t \wedge t_{R_h}^{h,\varrho}) + I_2(t \wedge t_{R_h}^{h,\varrho}) + I_3(t \wedge t_{R_h}^{h,\varrho}),
\quad\text{for all $ t \in [0,T] $},
\]
where, for any $ t \in [0,T] $,
\begin{align*}
I_1(t) &:= S_h(t)\mathcal{P}_hy_0, \\
I_2(t) &:= -\left[ S_h \ast \left(\mathbbm{1}_{[0,t_{R_h}^{h,\varrho}]} \mathcal{P}_hG(y_h)\right) \right](t), \\
I_3(t) &:= \left[ S_h \diamond \left(\mathbbm{1}_{[0,t_{R_h}^{h,\varrho}]} \mathcal{P}_hF(y_h)\right) \right](t).
\end{align*}
Using \cref{lem:Sh}(i), \cref{lem:Ph}(ii),
and the definition of $ t_{R_h}^{h,\varrho} $ in \cref{eq:tR-new},
a direct calculation yields the following estimate for $ I_1 $:
\begin{equation}
  \label{eq:I1-esti}
\sum_{j=0}^{J-1} \int_{{j\tau}}^{{j\tau+\tau}}
\left\lVert I_1(t\wedge t_{R_h}^{h,\varrho}) - I_1({j\tau}\wedge t_{R_h}^{h,\varrho})\right\rVert_{\dot{H}_h^{\alpha,2}}^p \, \mathrm{d}t
\leqslant c\tau^{1+\frac{\varrho-\alpha}{2}p}R_h^p, \quad\text{almost surely.}
\end{equation}
For \( I_2 \), the \( h \)-independent embeddings \( \dot{H}_h^{\varrho, 2} \)
into \( \mathbb{L}^{\infty} \) and \( \dot H_h^{1,2} \),
the boundedness of \( \mathcal{P}_h \) on \( \mathbb{L}^2 \),
 and the definition of \( G \) in \cref{eq:G-def} collectively imply that
\[
\left\lVert \mathbbm{1}_{[0,t_{R_h}^{h,\varrho}]}\mathcal{P}_hG(y_h) \right\rVert_{L^\infty(0,T;\dot{H}_h^{0,2})}
\leqslant c R_h^2, \quad\text{almost surely.}
\]
By \cref{lem:Sh}(i), an elementary calculation gives that
\[
  \norm{I_2}_{C^{1-3\alpha/4}([0,T];\dot H_h^{\alpha,2})}
  \leqslant cR_h^2, \quad\text{almost surely,}
\]
where $ C^{1-3\alpha/4}([0,T];\dot H_h^{\alpha,2}) $
denotes the space of $ \dot H_h^{\alpha,2} $-valued Hölder continuous functions with exponent
\( 1-3\alpha/4 \). Consequently,
\begin{equation}
  \label{eq:I2-esti}
    \sum_{j=0}^{J-1} \int_{{j\tau}}^{{j\tau+\tau}}
    \left\lVert I_2(t\wedge t_{R_h}^{h,\varrho}) - I_2({j\tau}\wedge t_{R_h}^{h,\varrho})\right\rVert_{\dot{H}_h^{\alpha,2}}^p \, \mathrm{d}t
  \leqslant c \tau^{(1-3\alpha/4)p}R_h^{2p},
  \quad\text{almost surely.}
\end{equation}
For \( I_3 \), we first note that \cref{lem:F}(iii) and \cref{lem:Ph}(vi) imply
\[
\left\lVert \mathbbm{1}_{[0,t_{R_h}^{h,\varrho}]} \mathcal{P}_h F(y_h) \right\rVert_{L^\infty(0,T;\gamma(H,\dot{H}_h^{(1+2\alpha)/4,2}))} \leqslant c R_h,
\quad\text{almost surely}.
\]
Similar to Theorem 5.15 in \cite{Prato2014}, using the factorization
formula (\cite[Theorem~5.10]{Prato2014}) and \cref{lem:Sh}(i) yields
\[
  \lVert I_3 \rVert_{L^p(\Omega; C^{1/2-1/p}([0,T]; \dot H_h^{\alpha,2}))} \leqslant c R_h.
\]
This leads to
\begin{align*}
  \mathbb{E} \, \Biggl[\,
    \sum_{j=0}^{J-1} \int_{j\tau}^{j\tau+\tau}
    \left\lVert I_3(t\wedge t_{R_h}^{h,\varrho}) - I_3(j\tau\wedge t_{R_h}^{h,\varrho})\right\rVert_{\dot{H}_h^{\alpha,2}}^p
    \, \mathrm{d}t
  \,\Biggr] \\
  \leqslant{} \mathbb{E} \, \Biggl[\,
    \sum_{j=0}^{J-1} \int_{j\tau}^{j\tau+\tau}
    \left\lVert I_3(t) - I_3(j\tau)\right\rVert_{\dot{H}_h^{\alpha,2}}^p \, \mathrm{d}t
  \,\Biggr] + c \tau^{\frac{p}{2}} R_h^p.
\end{align*}
Moreover, by Lemma \ref{lem:Sh}(i), a straightforward calculation yields
\[
  \mathbb{E} \left[\,
    \sum_{j=0}^{J-1} \int_{j\tau}^{j\tau+\tau} \left\lVert I_3(t) - I_3(j\tau) \right\rVert_{\dot{H}_h^{\alpha,2}}^p \, \mathrm{d}t
  \,\right] \leqslant c \tau^{\frac{p}{2}} R_h^p.
\]
Consequently, we derive the following estimate for \( I_3 \):
\begin{equation}
  \label{eq:I3-esti}
\mathbb{E}\Biggl[\,
\sum_{j=0}^{J-1} \int_{j\tau}^{j\tau+\tau}
\left\lVert I_3(t\wedge t_{R_h}^{h,\varrho}) - I_3(j\tau\wedge t_{R_h}^{h,\varrho}) \right\rVert_{\dot{H}_h^{\alpha,2}}^p
\, \mathrm{d}t
\,\Biggr] \leqslant c\tau^{\frac{p}{2}}R_h^p.
\end{equation}
Combining \cref{eq:I1-esti,eq:I2-esti,eq:I3-esti}, we derive the following bound:
\begin{align*}
\mathbb{E}\Biggl[\,
\sum_{j=0}^{J-1} \int_{j\tau}^{j\tau+\tau}
\norm{y_{h,R_h}(t) - y_{h,R_h}(j\tau)}_{\dot{H}_h^{\alpha,2}}^p \, \mathrm{d}t
\,\Biggr] \leqslant c\tau^{\frac{p}{2}}R_h^p \left( 1 + \tau^{\frac{2-3\alpha}4p} R_h^p\right).
\end{align*}
Similarly, we have
\begin{align*}
\mathbb{E}\Biggl[\,
\sum_{j=0}^{J-1} \int_{j\tau}^{j\tau+\tau}
\left\lVert y_{h,R_h}(t) - y_{h,R_h}(j\tau+\tau)\right\rVert_{\dot{H}_h^{\alpha,2}}^p \, \mathrm{d}t
\,\Biggr] \leqslant c\tau^{\frac{p}{2}}R_h^p \left( 1 + \tau^{\frac{2-3\alpha}4p} R_h^p\right).
\end{align*}
The combination of these inequalities establishes the desired inequality \cref{eq:yh-interp}, thereby completing the proof.
\end{proof}

\section{Conclusions}
\label{sec:conclusion}
This study presents a comprehensive analysis of a fully discrete numerical scheme for the two-dimensional stochastic Navier-Stokes equations with multiplicative noise, subject to no-slip boundary conditions. The spatial discretization employs the standard $P_3 / P_2$ Taylor-Hood finite element method, while the temporal discretization utilizes the Euler scheme. We have rigorously established the pathwise uniform convergence of the full discretization in probability, demonstrating nearly $3/2$-order spatial convergence and near half-order temporal convergence.

However, our study does not address the pressure field, and establishing comparable convergence results for it remains an open and interesting question (see Remark \ref{rem:pressure-conv}). Moreover, to the best of our knowledge, there is a lack of reported convergence rates under general spatial \( L^q \)-norms for finite-element-based approximations of the stochastic Navier-Stokes equations. We believe that conducting a thorough numerical analysis under these norms could yield valuable insights and represents a promising direction for future research.

Most significantly, establishing similar pathwise uniform convergence estimates in probability for the three-dimensional Navier-Stokes equations remains a significant challenge. A key reason for this is the widely acknowledged fact that global mild solutions generally do not exist for the three-dimensional case without the inclusion of additional regularization terms in the governing equations; see, for instance, \cite{Brzezniak2020,Rockner2009}.

      \end{document}